\documentclass[11pt, reqno]{amsart}
\usepackage[latin1]{inputenc}
\usepackage{amsmath}
\usepackage{amsfonts}
\usepackage{amssymb}
\usepackage{algorithmic}
\usepackage{algorithm}
\usepackage[arrow,matrix,curve,cmtip,ps]{xy}
\usepackage{amsthm}
\allowdisplaybreaks
\usepackage{textcomp}
\usepackage{comment}
\usepackage{hyperref}
\newtheorem{thm}{Theorem}
\newtheorem{prop}[thm]{Proposition}
\newtheorem{lem}[thm]{Lemma}
\newtheorem{cor}[thm]{Corollary}

\newtheorem{defn}[thm]{Definition}
\newtheorem{rmk}[thm]{Remark}

\newtheorem{eg}[thm]{Example}
\newcommand{\bb}[1]{\mathbb{#1}}
\newcommand{\cl}[1]{\mathcal{#1}}

\numberwithin{algorithm}{section} 
\numberwithin{thm}{section}
\numberwithin{equation}{section}
\numberwithin{table}{section}

\begin{document}
\title[Equiangular Frames and Signature Sets]{Equiangular Frames and Signature Sets}
\author{Preeti Singh}

\address{Department of Mathematics \\ University of Houston \\ Houston, TX 77204-3008 \\USA}
\email{presingh@math.uh.edu}
\begin{abstract}
We will present a relation between real equiangular frames and certain special sets in groups which we call signature sets and show that many equiangular frames arise in this manner. Then we will define quasi-signature sets and will examine equiangular frames associated to these subsets of groups. We will extend these results to complex equiangular frames where the inner product between any pair of vectors is a common multiple of a cube root of unity and exhibit equiangular frames that arise from groups in this manner.
\end{abstract}
\maketitle
\section {Introduction}
Equiangular tight frames play an important role in several areas of mathematics,
ranging from signal processing (see, e.g. \cite{port}, \cite{cas}, \cite{life1}, \cite{life2} and references therein) to
quantum computing.\\
The problem of the existence of equiangular frames is known to be equivalent
to the existence of a certain type of matrix called a Seidel matrix \cite{seidel} or signature
matrix \cite{holmes} with two eigenvalues. A matrix $Q$ is a Seidel matrix provided that
it is self-adjoint, its diagonal entries are $0$, and its off-diagonal entries are all of
modulus one. In the real case, these off-diagonal entries must all be $\pm 1$; such
matrices can then be interpreted as (Seidel) adjacency matrices of graphs. \\
In this paper we will use some basic facts in group theory to construct these signature matrices. We will look at some special subsets of groups which we will call signature sets and show that many equiangular frames arise from these sets. In Section \ref{sec2} we provide some necessary and sufficient conditions for the existence of the signature sets followed by some examples of equiangular frames arising from them.\\
 In \cite{kalra}, a relation between complex cyclic equiangular frames and difference sets was presented. In Section \ref{sec3}, we will establish a relation between signature sets and difference sets. We will see that a difference set having certain properties form a signature set for some special class of equiangular frames. \\
We will also study signature matrices $Q^{'}$ in the standard form where its diagonal entries are $0$, off-diagonal entries are all of modulus one and $Q_{i1}^{'}= Q_{1i}^{'}=1$. In Section \ref{sec4}, we will define quasi-signature sets and use these sets to construct signature matrices in the standard form and will show that many equiangular frames of the type $(2k,k)$ arise from quasi-signature sets.\\
Finally, we will extend our results in the case of real equiangular frames to the case when the entries of $Q$ are cube roots of unity. In \cite{cuberoots}, it was shown that the existence of such matrices is equivalent to the existence of certain highly regular directed graphs. In this paper we will use group theory and combinatorics to show the analogy of the results between this paper and \cite{cuberoots}.
\section{Definitions and Preliminaries}
 \subsection{Equiangular Tight Frames:} Let $\cl H$ be a real or complex Hilbert space. A finite family of vectors $\{f_1, . . . , f_n\}$ is called a frame provided that there exist
strictly positive real numbers $A$ and $B$ such that \begin{eqnarray}A\| x \| ^2 \leq 
\sum_{j=1}^n
|\langle x, f_j\rangle |^2 \leq B\| x \| ^2&& \text {for all $x\in \cl H$.}\end{eqnarray} 

A frame is said to be a tight frame if we can choose $A = B$. When $A = B = 1$,
then the frame is called a {\it normalized tight frame} or a {\it Parseval frame}. Replacing
$f_i$ by $f_i/\sqrt A$ always normalizes a tight frame.\\
One can show (see \cite{cas}, p. 21, for example) that a family $\{f_1, \dots , f_n\}$ is a tight
frame with constant $A$ if and only if
\begin{eqnarray}\label{1.1}x = \frac{1}{A}\sum_{i=1}^n\langle x,f_i\rangle f_i && \text{for all $x\in \cl H$.}\end{eqnarray} There is a natural equivalence relation for tight frames, motivated by simple operations on the frame vectors which preserve identity \eqref{1.1}.
$ $
We say that two tight frames $\{f_1, f_2,\dots,f_n\}$ and $\{g_1, g_2, \dots, g_n\}$ are \textit{unitarily equivalent} if there exists a unitary operator $U$ on $\cl H$ such that for all $i\in \{1, 2, \dots, n\}$, $g_i = Uf_i$. We say that they are \textit{switching equivalent} if there exist a unitary operator $U$ on $\cl H$, a permutation $\pi$ on $\{1, 2,\dots, n\}$ and a family of unimodular constants $\{\lambda_1, \lambda_2, \dots, \lambda_n\}$ such that for all $i \in \{1, 2, \dots, n\}$, $g_i = \lambda_iUf_{\pi_i}$. If $\cl H$ is a real Hilbert space, $U$ is understood to be orthogonal, and all $\lambda_i \in \{\pm 1\}$.\\
In this paper we shall be concerned only with Parseval frames for the $k$-dimensional complex Hilbert space $\bb C^k$, equipped with the canonical inner product. We use the term $(n, k)$-frame to mean a Parseval frame of $n$ vectors for $\bb C^k$. Every such Parseval frame gives rise to an isometric embedding of $\bb C^k$ into $\bb C^n$ via the map \begin{eqnarray*}V:\bb C^k\longrightarrow \bb C^n, & (V x)_j = \langle x, f_j\rangle, &\text{for all $j=\{1,2,\dots,n\}$}\end{eqnarray*} which is called the analysis operator of the frame. Because $V$ is linear, we may identify $V$ with an $n \times k$ matrix and the vectors $\{f_1,\dots, f_n\}$ are the respective columns of $V$. Conversely, given any $n \times k$ matrix $V$ that defines an isometry, if we let $\{f_1,\dots,f_n\}$ denote the columns of $V$, then this set is an $(n, k)$-frame and $V$ is the analysis operator of the frame. \\If $V$ is the analysis operator of an $(n, k)$-frame, then since $V$ is an isometry, we see that $V ^*V = I_k$ and the $n \times n$ matrix $V^*V$ is a self-adjoint projection of rank $k$. Note that $VV^*$ has entries $(VV^*)_{ij} = (\langle f_j , f_i\rangle)$. Thus, $VV^*$ is the Grammian matrix (or correlation matrix ) of the set of vectors. Conversely, any time we have an $n \times n$ self-adjoint projection $P$ of rank $k$, we can always factor it as $P = V V^*$ for some $n \times k$ matrix $V$. In this case we have $V^*V = I_k$ and hence $V$ is an isometry and the columns of $V$ are an $(n, k)$-frame. Moreover, if $P = WW^*$ is another factorization of $P$, then there exists a unitary $U$ such that $W^* = UV^*$, and the frame corresponding to $W$ differs from the frame corresponding to $V$ by applying the same unitary to all frame vectors, which is included in our equivalence relation. \\If there exists a unitary $U$ that is the product of a permutation and a diagonal unitary (orthogonal matrix in the real case) such that $UVV^{*}U^{*} = WW^{*}$, then the two frames corresponding to $V$ and $W$ are called {\it equivalent} as defined in \cite{holmes}.

\begin{defn}An $(n, k)$-frame $\{f_1, \hdots,f_n\}$ is called uniform if there is a constant
$u > 0$ such that $\|f_i\| = u$ for all $i$. An $(n, k)$-frame is called equiangular if all of the frame vectors are non-zero and the angle between the lines generated by any pair
of frame vectors is a constant, that is, provided that there is a constant $b$ such
that $|\langle f_i/\|f_i\|, f_j/\|f_j\|\rangle| = b$ for all $i \neq j$.
\end{defn}
Many places in the literature define equiangular to mean that the $(n, k)$-frame
is uniform and that there is a constant $c$ so that $\langle f_i, f_j\rangle = c$ for all $i\neq j$. However, the assumption that the frame is uniform is not needed in our definition as the
following result shows.

\begin{prop}
Let $\{f_1, \hdots,f_n\}$ be a tight frame for $\bb C_k$. If all frame vectors
are non-zero and if there is a constant $b$ so that $|\langle f_i/\|f_i\|, f_j/\|f_j\|\rangle| = b$ for all
$i\neq j$, then $\|f_i\| = \|f_j\|$ for every $i$ and $j$.
\end{prop}
\begin{proof}
 Without loss of generality, we may assume that the frame is a Parseval
frame, so that $P = (\langle f_j , f_i\rangle)_{i,j=1}^n$
is a projection of rank $k$. Hence, $P = P^2$ and so
upon equating the $(i, i)$-th entry and using the fact that the trace of $P$ is $k$, we see
that $\|f_i\|^2 = \langle f_i, f_i\rangle = \sum_{j=1}^n
\langle f_j , f_i\rangle \langle f_i, f_j\rangle = \|f_i^4\| + \sum_{j\neq i}^n b^2\|f_i\|^2\|f_j\|^2 = \|f_i\|^4 + b^2\|f_i\|^2(k - \|f_i\|^2)$, which shows that $\|f_i\|^2$ is a (non-zero) constant independent
of $i$. \\
\end{proof}
In \cite{holmes}, a family of $(n, k)$-frames was introduced that was called $2$-uniform
frames. It was then proved that a Parseval frame is $2$-uniform if and only if
it is equiangular. Thus, these terminologies are interchangeable in the literature,
but the equiangular terminology has become more prevalent.

\subsection{Seidel Matrices and Equiangular Tight Frames} 
\begin{defn}A matrix $Q$ is called a Seidel matrix provided that it is self-adjoint, its diagonal entries are $0$, and its off-diagonal entries are all of modulus $1$.
\end{defn} 

The previous section shows that an $(n, k)$-frame is determined up to unitary equivalence
by its Grammian matrix. This reduces the problem of constructing an
$(n, k)$-frame to constructing an $n \times n$ self-adjoint projection $P$ of rank $k$.
If an $(n, k)$-frame $\{f_1, f_2,\hdots,f_n\}$ is uniform, then it is known that $\|f_i\|^2 = \frac{k}{n}$
for all $i=\{1, 2,\hdots,n\}$. It is shown in (\cite{holmes}, Theorem 2.5) that if $\{f_1, \hdots, f_n\}$ is an
equiangular $(n, k)$-frame, then for all $i\neq j$, $|\langle f_j , f_i\rangle| = c_{n,k} = \sqrt{\frac{k(n-k)}{n^2(n-1)}}$.
Thus we may write \[V V ^* = (\frac{k}{n})I_n + c_{n,k}Q\] where $Q$ is a self-adjoint $n\times n$ matrix satisfying $Q_{ii} = 0$ for all $i$ and $|Q_{ij}|=1$ for all $i \neq j$. This matrix $Q$ is called the signature matrix associated with the $(n, k)$-equiangular frame.
$ $
The following theorem characterizes the signature matrices of equiangular $(n, k)$-
frames.
\begin{thm}\label{sigm}(Theorem 3.3 of \cite{holmes}). Let $Q$ be a self-adjoint $n \times n$ matrix with
$Q_{ii} = 0$ and $|Q_{ij} | = 1$ for all $i \neq j$. Then the following are equivalent:
\begin{enumerate}

\item $Q$ is the signature matrix of an equiangular $(n, k)$-frame for some $k$;
\item[(a)]\label{condsigm2} $Q^2 = (n - 1)I + \mu Q$ for some necessarily real number $\mu$; and
\item[(b)]\label{condsigm3} $Q$ has exactly two eigenvalues.
\end{enumerate}
\end{thm}
This result reduces the problem of constructing equiangular $(n, k)$-frames to
the problem of constructing Seidel matrices with two eigenvalues. In particular,
condition \eqref{condsigm2} is particularly useful since it gives an easy-to-check condition to
verify that a matrix $Q$ is the signature matrix of an equiangular tight frame.
Furthermore, if $Q$ is a matrix satisfying any of the three equivalent conditions
in Theorem \ref{sigm}, and if $\lambda_1 < 0 < \lambda_2$ are its two eigenvalues, then the parameters
$n$, $k$, $\mu$, $\lambda_1$, and $\lambda_2$ satisfy the following properties:
\begin{align}\label{relation}\mu = (n-2k)\sqrt{\frac{n-1}{k(n-k)}} = \lambda_1 + \lambda_2, && k =\frac{n}{2} -\frac{\mu n}{2 \sqrt{4(n-1) + \mu^2}}\end{align}
\begin{align}\lambda_1 = -\sqrt{\frac{k(n-1)}{n-k}}, &&\lambda_2 = \sqrt{\frac{(n-1)(n-k)}{k}}, &&  n = 1 -\lambda_1\lambda_2.\end{align}
These equations follow from the results in (\cite{holmes} Proposition 3.2) and (\cite{holmes}, Theorem
3.3), and by solving for $\lambda_1$ and $\lambda_2$ from the given equations.
In the case when the entries of $Q$ are all real, we have that the diagonal entries
of $Q$ are $0$ and the off-diagonal entries of $Q$ are $\pm 1$. In \cite{grass}, it has been noted that signature matrices of real equiangular frames are always Seidel adjacency matrices of regular two-graphs.

\begin{defn} Two Seidel matrices $Q$ and $Q^{'}$ are switching equivalent if they
can be obtained from each other by conjugating with a diagonal unitary and a
permutation matrix.
\end{defn}

\begin{prop}\label{switch}Let $Q$ be a signature matrix of an $(n,k)$-equiangular frame. If $Q$ is switching equivalent to a Seidel matrix $Q^{'}$, then $Q^{'}$ is also a signature matrix of an $(n,k)$-equiangular frame. 
\end{prop}
\begin{proof}Since $Q$ and $Q^{'}$ are switching equivalent, there exist a diagonal unitary $U$ and a permutation $P$ such that $Q^{'} = UPQP^{t}U^{*}$. Then \begin{align*}
Q^{{'}^2} & = (UPQP^{t}U^{*})^2\\
& = UPQ^2P^{t}U^{*}\\
& = UP((n-1)I + \mu Q)P^{t}U^{*} &&\text{from Theorem \ref{sigm}}\\
& = (n-1)I + \mu UPQP^{t}U^{*}\\
& = (n-1)I + \mu Q^{'}.\end{align*}
Again by using Theorem \ref{sigm}, $Q^{'}$ forms a signature matrix for an $(n,k)$-equiangular frame.
\end{proof}

\section{Real Equiangular Frames and Signature sets}\label{sec2}
Let $G$ be a finite group of order $n$. Let $\lambda: G \longrightarrow GL(\bb F(G))$ be the left regular representation such that 
$\lambda(g)e_h=e_{gh}$ where $\bb F(G)$ is the free vector space over $G$. Then we know that $\sum_{g\in G}\lambda(g)= J$ where $J$ is the $n\times n$ matrix of all $1's$.
\begin{defn}\label{defsigset}
Let $G$ be a group of order $n$ and $S \subset G\setminus\{e\}$, $T = S^c\setminus \{e\}$ such that $G\setminus \{e\} = S\cup T$.  Form  
$ Q = \sum_{g\in S}\lambda(g)- \sum_{h\in T}\lambda(h)$. Then $Q$ is an $n\times n$ matrix with $Q_{ii} = 0$ and $|Q_{ij} | = 1$ for all $i \neq j$. We call $S$ a signature set for an $(n,k)$-equiangular frame if $Q$ is a signature matrix for an $(n,k)$-equiangular frame. 
\end{defn}
\begin{prop}\label{prop1}Let $G$ be a finite group of order $n$ and $S\subset G$. If $S$ forms a signature set for an $(n,k)$-equiangular frame, then $T= S^c\setminus\{e\}$ forms a signature set for an $(n,n-k)$-equiangular frame. 
\end{prop}
\begin{proof} If $S$ forms a signature set for an $(n,k)$-equiangular frame, then $Q$ as in Definition \ref{defsigset} forms a signature matrix for $(n,k)$-equiangular frame.\\ Let $\tilde Q = \sum_{h\in T}\lambda(h)- \sum_{g\in S}\lambda(g)$. Then $\tilde Q = -Q$ and thus $\tilde Q$ is self adjoint with $\tilde Q_{ii}=0$ and for all $i\neq j$, $\tilde |Q_{ij}|=1$. Consider 
\begin{align*}\tilde Q^2 & = (-Q)^2\\
& = Q^2\\
& = (n-1)I +\mu Q && \text{by Theorem \ref{sigm}}\\
& = (n-1)I -\mu \tilde Q
.\end{align*}
Thus $\tilde Q$ forms a signature matrix for $(n,\tilde k)$-equiangular frame for some $\tilde k$. Using Equation \eqref{relation} for $-\mu$, we get $\tilde k =n-k$. Thus, $T$ forms a signature set for for an $(n,n-k)$-equiangular frame.  
\end{proof}
\begin{defn}
Given any subset $S$ of $G$, a subset $\tilde S$ of $G$ is said to be conjugate to $S$ if and only if there exists some $\tilde g$ in G such that  $\tilde S = \tilde g S \tilde g^{-1}$. 
\end{defn}

\begin{prop}\label{conj}Let $G$ be a finite group of order $n$ and $S\subset G$. If $S$ forms a signature set for an $(n,k)$-equiangular frame, then for any $\tilde g\in G$, the set $\tilde S = \tilde g S \tilde g^{-1}= \{\tilde g \cdotp g \cdotp\tilde g^{-1}: g\in S\}$ also forms a signature set for an $(n,k)$-equiangular frame.\end{prop}
\begin{proof} Let $T = S^c\setminus\{e\}$ and $\tilde T = \tilde S^c\setminus \{e\}$. Form $Q= \sum_{g\in S}\lambda(g)- \sum_{h\in T}\lambda( h)$ and $\tilde Q = \sum_{g\in \tilde S}\lambda(g)-\sum_{h\in \tilde T}\lambda(h)$. Then, 
\begin{align*}\tilde Q & = \sum_{g\in S}\lambda(\tilde g\cdotp g\cdotp\tilde g^{-1})-\sum_{h\in T}\lambda(\tilde g \cdotp h \cdotp\tilde g^{-1})\\
 & = \sum_{g\in S}\lambda(\tilde g) \lambda (g)\lambda(\tilde g^{-1})-\sum_{h\in T}\lambda(\tilde g)\lambda( h)\lambda(\tilde g^{-1})\\
& = \lambda(\tilde g)( \sum_{g\in S}\lambda(g)- \sum_{h\in T}\lambda( h)) \lambda(\tilde g^{-1})\\
& = \lambda(\tilde g)Q \lambda(\tilde g^{-1}).\end{align*} 
$Q$ self adjoint implies that $\tilde Q$ is self adjoint. Also since for all $g\in G$, $\lambda(g)$ is a permutation, from Proposition \ref{switch}, $\tilde Q$ also forms a signature matrix for an $(n,k)$-equiangular frame and thus $\tilde S$ forms a signature set for an $(n,k)$-equiangular frame. 
\end{proof}

\noindent Given subsets $A,B \subseteq G$ and $g\in G$, we define \[N_{(A,B)}^g = \#\{(g_1,g_2)\in A\times B: g_1\cdot g_2=g\}.\]
$ $
 
\begin{lem}\label{lem1}Let $G$ be a finite group and $S,T \subset G\setminus \{e\}$ be disjoint such that $G\setminus\{e\}= S\cup T$. Then for all $g\in G$, $N_{(S,T)}^g = N_{(T,S)}^g$.\end{lem}
\begin{proof}Let $|G| =n$ and $|S|= m$. Then $|T|= n-m-1$. For $g\in S$, assume that $N_{(S,T)}^g= l$. Then there are $l$ ordered pairs $(g_i,h_i)\in S\times T$ such that for all $i\in \{1,\dots,l\}$, $g_i\cdot h_i = g$. Let us order the elements of $S$ as $\{g,g_1,\dots,g_l,g_{l+1},\dots,g_{m-1} \}$. Thus for all $i\in \{l+1,\dots, m-1\}$, we have $g_i\cdot g_j = g$ for some $g_j\in S$, $g_j\neq g$. That is $N_{(S,S)}^g = m-1-l$.  Again if we order the elements of $T$ as $\{h_1,\dots,h_l,h_{l+1},\dots, h_{n-m-1}\}$, then for all $i\in \{l+1,\dots, n-m-1\} $, we have $h_j\cdot h_i = g$ for some $h_j\in T$. Thus $N_{(T,T)}^g= n-m-1-l$. Since $N_{(T,T)}^g+ N_{(T,S)}^g = |T| =n-m-1 $, we have $N_{(T,S)}^g= l$. 
Similarly we can prove that for all $h\in T$, $N_{(S,T)}^h = N_{(T,S)}^h$.
\end{proof}
\vspace{0.2in}
Following is a necessary and sufficient condition for a set $S$ in $G$ to be a signature set for an $(n,k)$-equiangular frame. 
\begin{thm}\label{mainthm1} Let $G$ be a finite group of order $n$ and $S,T \subset G\setminus \{e\}$ where $T = S^c\setminus \{e\}$ such that $G\setminus \{e\} = S\cup T$. Then there exists a $k$ such that $S$ is a signature set for an $(n,k)$-equiangular frame if and only if the following hold: 
\begin{enumerate}
\item[(a)] $g\in S$ implies $g^{-1} \in S$ and $h\in T$ implies $h^{-1} \in T$;
\item[(b)] there exists a real number $\mu$ such that for all $g \in S$;
\begin{equation}\label{cond1}N_{(S,S)}^g -2 N_{(S,T)}^g  +N_{(T,T)}^g  = \mu ;\end{equation} and 
for all $h\in T$, \begin{equation}\label{cond2}N_{(S,S)}^h - 2N_{(S,T)}^h + N_{(T,T)}^h  = -\mu.\end{equation}
\end{enumerate}
In this case $k$ and $\mu$ are related by the Equations in \eqref{relation}.\end{thm}
\begin{proof}
Form $Q =\sum_{g\in S}\lambda(g)- \sum_{h\in T}\lambda(h)$. Then by Definition \ref{defsigset}, $S$ will form a signature set for an $(n,k)$-equiangular frame if and only if $Q$ forms a signature matrix for an $(n,k)$-equiangular frame. From Theorem \ref{sigm} we know that an $n\times n$ matrix $Q$ with $Q_{ii}=0$ and for all $i\neq j$, $|Q_{ij}|=1$, forms a signature matrix for an $(n,k)$ equiangular frame if and only if it satisfies the following two conditions:
\begin{enumerate}
\item[(a)] $Q$ is self adjoint that is $Q = Q^*$; and 
\item[(b)] $Q^2 = (n-1)I + \mu Q$ for some real number $\mu$.
\end{enumerate}
The condition $Q = Q^*$ is equivalent to \begin{align*}\sum_{g\in S}\lambda(g)- \sum_{h\in T}\lambda(h)& = (\sum_{g\in S}\lambda(g)- \sum_{h\in T}\lambda(h))^*\\
& = \sum_{g\in S}\lambda(g^{-1})- \sum_{h\in T}\lambda(h^{-1}).\end{align*} 
Thus $g\in S$ implies $g^{-1}\in S$ and $h\in T$ implies $h^{-1} \in T$.
The second condition $Q^2 = (n-1)I+ \mu Q$, for some real number $\mu$, is equivalent to 
\begin{equation*}
\begin{split}
& \sum_{\substack{
g_1,g_2\in S\\
g_1\neq g_2}}\!\!\lambda(g_1\cdot g_2) -\sum_{\substack{
g_1\in S\\
h_1\in T}}\lambda(g_1\cdot h_1)-\sum_{\substack{
g_1\in S\\
h_1\in T}}\lambda(h_1\cdot g_1) +\sum_{\substack{
h_1,h_2\in T\\
h_1\neq h_2}}\lambda(h_1\cdot h_2)\\
 & = (n-1)I + \mu (\sum_{g\in S}\lambda(g)- \sum_{h\in T}\lambda(h)).
\end{split}
\end{equation*}
By counting arguments, we have $Q^2 = (n-1)I + \mu Q$, for some real number $\mu$, if and only if for all $g\in S$, $N_{(S,S)}^g - N_{(S,T)}^g- N_{(T,S)}^g+N_{(T,T)}^g  = \mu $ and 
for all $h\in T$, we have $N_{(S,S)}^h -N_{(S,T)}^h-N_{(T,S)}^h  +N_{(T,T)}^h  = -\mu $. Using Lemma \ref{lem1}, we have $Q^2 = (n-1)I + \mu Q$, for some real number $\mu$, if and only if for all $g\in S$, $N_{(S,S)}^g - 2N_{(S,T)}^g+N_{(T,T)}^g  = \mu $ and 
for all $h\in T$, we have $N_{(S,S)}^h -2N_{(S,T)}^h+N_{(T,T)}^h  = -\mu$.
\end{proof}
\begin{rmk}From the relations given in the Equation \eqref{relation}, since $k$ is a function of $\mu$, we shall often use the parameter $\mu$ to specify our frames and denote them as $(n,k(\mu))$-equiangular frames.
\end{rmk}
Using some counting arguments, we can further simplify conditions \eqref{cond1} and \eqref{cond2} of Theorem \ref{mainthm1} given in the following result.
\begin{thm}\label{mainthm2}Let $G$ be a group with $|G|= n$. Let $S,T \subset G\setminus\{e\}$ where $T = S^c\setminus \{e\}$ such that $G\setminus \{e\} = S\cup T$. Also let $S = S^{-1}$ and $T = T^{-1}$. Then there exists $\mu$ such that $S$ forms a signature set for an $(n,k(\mu))$-equiangular frame if and only if for all $g\in S$ \begin{equation}\label{cond3}N_{(S,T)}^g=\frac {n-2-\mu}{4}\end{equation} and for all $h\in T$, \begin{equation}\label{cond4}N_{(S,T)}^h=\frac {n-2+\mu}{4}.\end{equation}\end{thm}
\begin{proof}Assume $|S|=l$. Since $|G|=n$, we have $|T|=n-1-l$. For $g\in S$, let $N_{(S,S)}^g=m$, then since $|S| = N_{(S,S)}^g+  N_{(S,T)}^g+1$, we have $N_{(S,T)}^g = l-1-m$. Also by Lemma \ref{lem1}, we have $N_{(S,T)}^g = N_{(T,S)}^g$ and using $|T| = N_{(T,T)}^g + N_{(T,S)}^g $, we have $N_{(T,T)}^g=n-1-l-(l-1-m)=n-2l+m$.\\ By Theorem \ref{mainthm1}, $S$ will form a signature set for an $(n,k(\mu))$ equiangular frame if and only if Equations \eqref{cond1} and \eqref{cond2} hold. That is for all $g\in S$, $N_{(S,S)}^g - 2N_{(S,T)}^g+N_{(T,T)}^g  = \mu $ and for all $h\in T$, $N_{(S,S)}^h - 2N_{(S,T)}^h+N_{(T,T)}^h  = -\mu $. $N_{(S,S)}^g - 2N_{(S,T)}^g+N_{(T,T)}^g  = \mu $ is equivalent to \begin{align*}
\mu & = m -2(l-1-m)+n-2l+m \\
& = 4m -4l +n +2.\end{align*} Thus, $l-m = \frac {n+2-\mu}{4}$. Since $N_{(S, T)}^g= l-m-1 $, Equation \eqref{cond1} holds if and only if for all $g\in S$, \[N_{(S, T)}^g= \frac {n-2-\mu}{4}.\]
Similarly, for $h\in T$, if $N_{(S, S)}^h = \tilde m$, then using $|S| = N_{(S,S)}^h+ N_{(S,T)}^h $, we have $N_{(S, T)}^h=l-\tilde m$. Also since $|T|=N_{(S, T)}^h+ N_{(T,T)}^h+1$, we have $N_{(T,T)}^h= n-2-l-(l-\tilde m) = n-2-2l+\tilde m$. The condition $N_{(S,S)}^h - 2N_{(S,T)}^h+N_{(T,T)}^h  = -\mu $ is equivalent to \begin{align*}-\mu & = \tilde m- 2(l-\tilde m) + n-2-2l+\tilde m\\
& = n-4l+4\tilde m-2.\end{align*}

Thus we have $l-\tilde m = \frac {n-2+\mu}{4}$. Since $N_{(S, T)}^h=l-\tilde m$, Equation \eqref{cond2} holds if and only if for all $h\in T$ we have, \[N_{(S, T)}^h= \frac {n-2+\mu}{4}.\]
\end{proof}
\begin{cor}\label{even}Let $G$ be a group of order $n$. If there exists a signature set $S$ in $G$ corresponding to an $(n,k(\mu))$-equiangular frame, then the following hold: 
\begin{enumerate}
\item[(a)] $n \equiv 0 \pmod 2$;
\item[(b)] $\mu \equiv 0\pmod2$;
\item[(c)]\label{3}$n$, $\mu$ satisfies $-(n-2)\leq \mu\leq (n-2)$.
\end{enumerate}\end{cor} 
\begin{proof}
From Theorem \ref{mainthm2}, if $S$ forms a signature set for an $(n,k(\mu))$-equiangular frame, then Equations \eqref{cond3} and \eqref{cond4} hold. If we sum the Equations \eqref{cond3} and \eqref{cond4}, we have $n = 0 \pmod 2$ and subtracting \eqref{cond3} from \eqref{cond4} gives us $\mu \equiv 0\pmod2$.\\
Since $N_{(S,T)}^g,N_{(S,T)}^h\geq 0$, again using \eqref{cond3} and \eqref{cond4}, we have $-(n-2)\leq \mu\leq(n-2)$.

\end{proof}
Using \eqref{relation} and the relations that we have proved in Corollary \ref{even}, we will now classify some of the $(n,k)$-equiangular frames arising from signature sets by looking at specific values of $n$ and $\mu$.
\begin{prop}\label{class}Let $G$ be a group and $S\subset G$ be a signature set for an $(n,k(\mu))$-equiangular frame, then the following hold. 
\begin{enumerate}
\item[(a)] If $\mu =0$, then $n=2m$ where $m \in \bb N$ is an odd number and $S$ forms a signature set for a $(2m,m)$-equiangular frame.
\item[(b)] If $\mu=2$, then $n=4a^2$ where $a\in \bb N$ and $S$ forms a signature set for $(4a^2,2a^2-a)$-equiangular frame.
\item[(c)]If $\mu = -2$, then $n=4a^2$ where $a\in \bb N$ and $S$ forms a signature set for $(4a^2,2a^2+a)$-equiangular frame.
\item[(d)] If $n=2p$ where $p$ is an odd prime, then either $\mu =0$ and $S$ forms a signature set for a $(2p,p)$-equiangular frame or $\mu=n-2$ and $S$ forms a signature set for an $(2p,1)$-equiangular frame.
\item[(e)]If $n=4p$ where $p$ is an odd prime, then $\mu=n-2$ and $S$ forms a signature set for an $(4p,1)$-equiangular frame.
\end{enumerate}
 \end{prop}
\begin{proof}
If $\mu =0$, then from Theorem \ref{mainthm2}, for all $g\in S$ and for all $h\in T$, we have $N_{(S, T)}^g = N_{(S, T)}^h = \frac{n-2}{4}$. Thus $n\equiv 2\pmod 4$ that is $n=4l+2$ for $l\in \bb N$ or equivalently $n=2m$ where $m\in \bb N$ is an odd number. Using Equation \eqref{relation}, we have $k=m$ and thus $S$ forms a signature set for $(2m,m)$-equiangular frame.\\
To prove the remaining parts, using Corollary \ref{even}, we can assume that $n=2n_1$ and $\mu=2\mu_1$ where $n_1 \in \bb N$ and $\mu_1\in \bb Z$. Using Equation \eqref{relation}, we get 
\begin{equation}\label{ineq1}
 k = \frac{2n_1}{2}-\frac{2\mu_1\cdotp 2n_1}{2\sqrt{4(2n_1-1)+4\mu_1^2}} = n_1-\frac{n_1\mu_1}{\sqrt{(2n_1-1)+\mu_1^2}}.\end{equation}Thus $n_1^2\mu_1^2\equiv 0\pmod {(2n_1-1+\mu_1^2)}$. If $\mu_1=\pm 1$, then we have $\mu=\pm 2$. Using Equation \eqref{ineq1}, $n_1$ must be of the form $n_1=2a^2$ where $a\in \bb N$ that is and $n=4a^2$ where $a\in \bb N$. Using Equation \eqref{relation} for $\mu =2$, we get $k= 2a^2-a$ and $S$ forms a signature set for $(4a^2,2a^2-a)$-equiangular frame. Again using \eqref{relation} for $\mu =-2$, we get $k= 2a^2+a$ and $S$ forms a signature set for $(4a^2,2a^2+a)$-equiangular frame.\\
If $n=2p$ that is $n_1=p$, then $n_1^2\mu_1^2\equiv 0\pmod{2n_1-1+\mu_1^2}$ implies that $\mu_1^2-1\equiv 0\pmod p$. If $\mu_1 =0$, then $\mu =0$ and by part (a), $S$ forms a signature set for $(2p,p)$-equiangular frame. If $\mu_1\neq 0$, then $\mu_1^2-1\equiv 0\pmod p$ implies that either $\mu_1-1\equiv 0\pmod p$ or $\mu_1+1\equiv 0\pmod p$. But from Corollary \ref{even}, part (c), we have $-(p-1)\leq \mu_1\leq (p-1)$. Thus we have $\mu_1=p-1$ that is $\mu = 2p-2 =n-2$. From Equation \eqref{relation}, we have $k=1$. Thus $S$ forms a signature set for an $(2p,1)$-equiangular frame.\\ Similarly if $n=4p$ that is $n_1=2p$, then $n_1^2\mu_1^2\equiv 0\pmod{2n_1-1+\mu_1^2}$ implies that $\mu_1^2-1\equiv 0\pmod {4p}$. By part (a) again, $\mu \neq 0$. Thus $\mu_1^2-1\equiv 0\pmod{4p}$ implies that $\mu_1-1\equiv 0\pmod 2$ and $\mu_1+1 \equiv 0\pmod 2$. Let $\mu_1+1= 2a$ for some $a\in \bb N$. Then $\mu_1-1 = 2a-2$ and $\mu_1^2-1\equiv 0\pmod{4p}$ implies $a(a-1)\equiv 0\pmod p$. Thus either $a\equiv 0\pmod p$ or $a-1\equiv 0\pmod p$. Again from Corollary \ref{even}, part (c), we have $-(2p-1)\leq \mu_1\leq (2p-1)$ that is $2-2p\leq 2a\leq 2p$ or equivalently we have $1-p\leq a\leq p$. Thus $a=p$ and $\mu_1=2p-1$ that is $\mu=2\mu_1 = 4p-2 = n-2$. Again we get $k= 1$ and thus $S$ forms a signature set for an $(4p,1)$-equiangular frame.

\end{proof}
Our goal now is to look for the signature sets in a group $G$. The first subsets we will look for in the groups are the subgroups. The following result characterizes the frames we get when we take $S$ to be a subgroup of $G$.
\begin{thm}\label{subgp}Let $G$ be a group of order $n$ and $H$ a proper subgroup of $G$. Then $H\setminus \{e\}$ is a signature set of an $(n,k(\mu))$ equiangular frame if and only if $H$ is a subgroup of index $2$.\\ In this case $\mu=n-2$ and thus $k=1$. 
\end{thm}
\begin{proof}Let $H$ be a subgroup of $G$. Let $S = H\setminus \{e\}$ and $T = H^c$. Then for all $g\in S$, we have $N_{(S,T)}^g= 0$ and for all $h \in T$, we have $N_{(S,T)}^h= |S|$. From Equations \eqref{cond3} and \eqref{cond4}, $S$ forms a signature set for an $(n,k(\mu))$ equiangular frame if and only if \begin{equation*} 0 = \frac {n-2-\mu}{4}
\end{equation*} and \begin{equation*} |S| = \frac {n-2+\mu}{4}\end{equation*}
which gives us $\mu = n-2$ and $2|S| = n-2$.\\
Since $|H| = |S|+1$, we have $|H| = \frac{n-2}{2}+1 = \frac{n}{2}$. Thus, $H$ is a subgroup of index $2$.\\
When $\mu = n-2$, using Equations \eqref{relation}, we get $k=1$. Hence we get $(n,1)$-equiangular frame.
\end{proof}
\begin{rmk}\label{rmk2} If $S = G\setminus \{e\}$, then $T = \emptyset$. Thus $Q= \sum_{g\in S}\lambda(g) =(J-I_n)$ where $J$ is the matrix of all $1's$. This gives us the trivial $(n,1)$-equiangular frame. So by Theorem \ref{subgp} we have one more way to get the trivial $(n,1)$-equiangular frame by taking subgroup of index $2$ in the group $G$ as the signature set.
\end{rmk}
\begin{rmk}By Proposition \ref{prop1}, the following subsets $S$ of $G$ are signature sets for the $(n,n-1)$-equiangular frame:
\begin{enumerate}
\item[(a)]$S = \emptyset$ (by Remark \ref{rmk2});
\item[(b)]$S= aH$ where $H$ is a subgroup of index $2$ in $G$ and $a \notin H$ (by Theorem \ref{subgp}).
\end{enumerate} 
\end{rmk}
So far we have seen the case of trivial equiangular frames only. Following propositions gives us some of the non-trivial equiangular frames arising from signature sets in groups of the form $C_n\times C_n$ where $C_n$ is a cyclic group of order $n$.
 
\begin{prop}\label{propeg1}Let $G\cong C_n\times C_n=\langle a,b:a^n=e, b^n=e, ab=ba\rangle$ and let $S=\{a, a^2,\hdots , a^{n-1}, b, b^2,\hdots, b^{n-1}\}$. Then $S$ forms a signature set for an $(n^2,k)$-equiangular frame if and only if either $n=2$ and $k=3$ or $n=4$ and $k=6$.
\end{prop} 
\begin{proof}$|S|= 2(n-1)$ and $|T|= n^2-1-2(n-1)= n^2-2n +1$. For all $g\in S$, we have $N_{(S,S)}^g=n-2$. Thus for all $g \in S$, $N_{(S,T)}^g= 2(n-1)-1-(n-2)= n-1$. Similarly for all $h\in T$, we have $N_{(S,S)}^h=2$. Thus for all $h \in T$, we have $N_{(S, T)}^h= 2(n-1)-2 = 2n-4 $. Using Equations \eqref{cond3} and \eqref{cond4}, $S$ will form a signature set for an $(n^2,k)$ equiangular frame if and only if 
 \[n-1= \frac {n^2-2-\mu}{4}\] and \[2n-4= \frac {n^2-2+\mu}{4}\]
which implies $\mu = n^2-4n+2$ and $\mu = -n^2+8n-14$.
Thus solving for $n$ we get $n=4$ or $n=2$.
For $n=4$, we have $\mu = 2$ and for $n=2$, we have $\mu =-2$.
Thus the equiangular frames that we get are $(16,6)$ and $(4,3)$ equiangular frames.
\end{proof}
\begin{prop}\label{propeg2}Let $G\cong C_n\times C_n=\langle a,b:a^n=e, b^n=e, ab=ba\rangle$ and let $S=\{a,\hdots , a^{n-1}, b,\hdots, b^{n-1},ab,\hdots,a^{n-1}b^{n-1}\}$. Then $S$ forms a signature set for an $(n^2,k)$-equiangular frame if and only if either $n=4$ and $k=10$ or $n=6$ and $k=15$.
\end{prop}
\begin{proof}
$|S|= 3(n-1)$ and $|T|= n^2-1-3(n-1)= n^2-3n +2$. For all $g\in S$, we have $N_{(S,S)}^g= n$. Thus for all $g \in S$, $N_{(S, T)}^g= 3(n-1)-1-n = 2n-4$. Similarly for all $h\in T$, we have $N_{(S,S)}^h=6$. Thus for all $h \in T$, $N_{(S, T)}^h= 3(n-1)-6 = 3n-9$. By Theorem \ref{mainthm2} and Equations \eqref{cond4}, $S$ will form a signature set of an $(n^2,k)$ equiangular frame if and only if \[2n-4\ = \frac {n^2-2-\mu}{4}\] and \[3n-9 = \frac {n^2-2+\mu}{4}.\] 
Thus solving for $n$ we get $n^2-8n+14 = -n^2+12n-34$ which gives us $n^2-10n+24 =0$. Thus either $n=4$ or $n=6$. For $n=6$, we have $\mu = 2$ and for $n=4$, we have $\mu =-2$. Thus the frames that we get are $(36,15)$ and $(16,10)$ equiangular frames.\end{proof}
$ $
A real $n \times n$ matrix $H$ is called a {\bf Hadamard matrix} \cite{conway} provided that $h_{i,j} = \pm 1$ and $H^{*}H = nI$. 
$ $
\begin{rmk}
 An explicit construction of signature matrices using the signature sets defined in the above two propositions has been shown in \cite{thesis}. Let us denote the signature matrices obtained from the Propositions \ref{propeg1} and \ref{propeg2} by $Q_1$ and $Q_2$ respectively. Then from Example 3.8 in \cite{holmes}, we infer that the matrix $I-Q_i$ is a Hadamard matrix for $i=1,2$.\end{rmk}

\section{Signature Sets and Difference Sets}\label{sec3}
In \cite{kalra}, a relation between cyclic difference sets and complex equiangular cyclic frames was shown. In this section we will present a relation between the two type of subsets: signature sets and difference sets in a group $G$. First we recall some of the basic facts about difference sets.
\begin{defn} Let $G$ be an additively written group of order $n$. A subset $D$ of $G$ with $|D|=k$ is a $(n,k,\lambda)$-difference  set of $G$ if for for some fixed number $\lambda$, every non zero element of $G$ can be written as a difference of two elements of $D$ in exactly $\lambda$ ways.
 \end{defn}

\begin{eg}
 The set $\{1,3,4,5,9\}$ is a $(11,5,2)$-difference set in $\bb Z_{11}$.
\end{eg}
\begin{eg}\label{eg1dif}
 Consider the signature set that we are getting in Proposition \ref{propeg1}. For $n=4$ we have $G = \bb Z_4 \times \bb Z_4$ and \[S = \{(1,0),(0,1),(2,0),(0,2),(3,0),(0,3)\}\] Then $S$ is also a $(16,6,2)$ difference set with $\lambda = 2$.
\end{eg}
\begin{eg}\label{eg2dif}
 Consider the signature set that we are getting in Proposition \ref{propeg2}. For $n=6$ we have $G = \bb Z_6 \times \bb Z_6$ and \[S = \{(1,0),(2,0), \dots,(5,0),(0,1),(0,2),\dots,(0,5), (1,1),(2,2),\hdots, (5,5)\}\] Then $S$ is also a $(36,15,3)$ difference set. 
\end{eg}
\begin{prop}\label{propdif}If $D$ is a $(n,k,\lambda)$ difference set in $G$, then the following hold:
 \begin{enumerate}
\item[(a)] $\lambda= \frac{k(k-1)}{n-1}$;
\item[(b)]$D^c$ is a $(n,n-k,\bar \lambda)$ difference set where $\bar \lambda = \frac{(n-k)(n-k-1)}{n-1}$.
 \end{enumerate}
\end{prop}
\begin{proof}
 Since $|D|=k$, the number of ordered pairs $(x,y)\in D\times D$ such that $x\neq y$ is equal to $k(k-1)$. On the other hand, $D$ has $n-1$ non-zero elements, and for each non-zero element $a\in G$, there are $\lambda$ ordered pairs $(x,y)\in D\times D$ such that $a=x-y$. Hence $k(k-1)=\lambda (n-1)$. \\
Since every non-zero element in $G$ can be written as a difference of two elements of $D$ in exactly $\lambda$ ways, it follows that every non-zero element in $G$ can be written as a difference of an element of $D$ and $D^c$ in exactly $k-\lambda$ ways. Thus every non-zero element in $G$ can be written as a difference of elements of $D^c$ in exactly $\bar\lambda =n-k-(k-\lambda)$ ways. Using the value of $\lambda= \frac{k(k-1)}{n-1}$, we get \[ \bar \lambda =n-k-(k-\lambda)= \frac{(n-k)(n-k-1)}{n-1}.\]
Thus $D^c$ is a $(n,n-k,\bar \lambda)$ difference set where $\bar \lambda = \frac{(n-k)(n-k-1)}{n-1}$.
\end{proof}
In \cite{davis}, a difference set $D$ in a group $G$ is called {\bf reversible} if $-D= \{-d:d\in D\}=D$.
\begin{rmk}\label{difsig}
 Let $D$ be a reversible $(n,k,\lambda)$ difference set in an additive group $G$. Then for any $g\in G$, \[\#\{(g_1,g_2)\in D\times D: g_1+g_2 = g\} = \#\{(g_1,g_2)\in D\times D: g_1-g_2 = g\}.\] In addition, for all $g\in G$, we have
\begin{enumerate}
 \item[(a)] $N_{(D,D)}^g = \lambda$;
\item[(b)] $N_{(D^c,D^c)}^g = \bar\lambda$;
\item[(c)] $N_{(D,D^c)}^g = c$, for some real $c$.
\end{enumerate} 
\end{rmk}
\vspace{0.05in}
\begin{lem}\label{difcard}Let $D$ be a $(n,k,\lambda)$ reversible difference set in a group $G$ such that $0\notin D$. Let $T=D^c\setminus\{0\}$, then for all $g\in D$ and for all $h\in T$, the following hold:
\begin{enumerate}
 \item[(a)] $N_{(D,T)}^g+1 = N_{(D,T)}^h$;
\item[(b)] $N_{(T,T)}^g = N_{(D^c,D^c)}^g = \bar\lambda$;
\item[(c)]$N_{(T,T)}^h +2 = N_{(D^c,D^c)}^h$.\end{enumerate}
\end{lem}
\begin{proof}
 Since $D$ is a $(n,k,\lambda)$ difference set, then by Proposition \ref{propdif}, part (b), $D^c$ is also a difference set. Since $0\in D^c$, for every $g\in D$, we have $(g,0)\in \{(g_1,g_2)\in D\times D^c: g_1-g_2 = g\}$. Thus $N_{(D,T)}^g = N_{(D,D^c)}^g-1$. But for any $h\in T$ we have $N_{(D,T)}^h = N_{(D,D^c)}^h$. Using Remark \ref{difsig}, we have $N_{(D,T)}^g +1 = N_{(D,T)}^h$. \\
For $g\in D$, if $g= h_1+h_2$, $h_1,h_2 \in D^c$, then $h_1\neq 0$ and $h_2\neq 0$. Thus $N_{(T,T)}^g = N_{(D^c,D^c)}^g = \bar\lambda$. \\
If $h\in T$, then $(h,0), (0,h) \in N_{(D^c,D^c)}^h = \#\{(h_1,h_2)\in D^c\times D^c: h_1-h_2 = h\}$. Since $0\notin T$, we have $N_{(T,T)}^h +2 = N_{(D^c,D^c)}^h$.\end{proof}
\vspace{0.05in}
The following result gives us a relation between the difference sets and the signature sets:
\begin{thm}\label{mainthmdif}Let $G$ be a group of order $n$ and $D$ be a $(n,k,\lambda)$ difference set in $G$.
\begin{enumerate}
 \item[(a)] If $0\notin D$, then $D$ forms a signature set for an $(n,k)$-equiangular frame if and only if $D$ is reversible and $k=\frac{n-\sqrt n}{2}$. 
\item[(b)]If $0\in D$, then $D\setminus\{0\}$ forms a signature set for an $(n,k)$-equiangular frame if and only if $D$ is reversible and $k=\frac{n+\sqrt n}{2}$.
\end{enumerate}\end{thm}
\begin{proof} Since $D$ is a $(n,k,\lambda)$ difference set in $G$, from Remark \ref{difsig} and Proposition \ref{propdif}, we have for all $g\in G$, $N_{(D,D)}^g= \lambda$ and $N_{(D^c,D^c)}^g = \bar \lambda$. For part (a), let us first assume that $D$ forms a signature set for an $(n,k)$-equiangular frame. Let $T=D^c\setminus\{0\}$. From Theorem \ref{mainthm2}, Equations \eqref{cond3} and \eqref{cond4} hold. That is we have for all $g\in D$ \begin{equation*}N_{(S,T)}^g=\frac {n-2-\mu}{4}\end{equation*} and for all $h\in T$, \begin{equation*}N_{(S,T)}^h=\frac {n-2+\mu}{4}.\end{equation*} Also $D$ signature set for an $(n,k)$-equiangular frame implies that $D$ is reversible. Using Lemma \ref{difcard}, part (a), we have for all $g\in D$ and for all $h\in T$, $N_{(D,T)}^g+1 = N_{(D,T)}^h$. Thus we have \begin{equation*} \frac{n-2-\mu}{4}+1= \frac{n-2+\mu}{4}.\end{equation*} Solving for $\mu$ we get $\mu =2$ and using \eqref{relation}, we get $k=\frac{n-\sqrt n}{2}$.\\ 
Conversely, assume that $D$ is reversible and $k=\frac{n-\sqrt n}{2} $. We claim that $D$ forms a signature set for an $(n,k)$-equiangular frame. For $g\in S$, we have \begin{align*}
       |G|-2& = N_{(D,D)}^g+2N_{(D,T)}^g+N_{(T,T)}^g\\
& = \lambda+ 2N_{(D,T)}^g+ \bar \lambda && \text{by Lemma \ref{difcard}, part (b)}.\end{align*}
Thus we get $n-2 = \lambda+ 2N_{(D,T)}^g+ \bar \lambda$. Using Proposition \ref{propdif} and $k= \frac{n-\sqrt{n}}{2}$, we get 
\begin{align*}
 2N_{(D,T)}^g & = (n-2 )-\frac{n-2\sqrt n}{4} - \frac{n+2\sqrt n}{4}\\
& = \frac{n}{2}-2.\end{align*}
Thus for all $g\in D$, $N_{(D,T)}^g = \frac{n}{4}-1$ and from Lemma \ref{difcard}, we have $N_{(D,T)}^g+1 = N_{(D,T)}^h$. Thus for all $h\in T$, we have $N_{(D,T)}^h = \frac{n}{4}$. By Theorem \ref{mainthm2}, $D$ forms a signature set for $(n,k)$-equiangular frame.\\
For part (b), since $D^c$ is a $(n,n-k,\bar\lambda)$ difference set with $0\notin D^c$, using the same argument as in part (a) for $D^c$, we have that $D^c$ forms a signature set for an $(n,n-k)$-equiangular frame if and only if $D^c$ is reversible and $n-k = \frac{n-\sqrt n}{2}$ that is $k= \frac{n+\sqrt n}{2}$. Using Proposition \ref{prop1}, $D\setminus\{0\}$ forms a signature set for $(n,k)$-equiangular frame if and only if $D$ is reversible and  $k= \frac{n+\sqrt n}{2}$.
\end{proof}
\begin{rmk}
Note that in the above result, part (a) we are getting $k=\frac{n-\sqrt n}{2}$. Using \eqref{relation}, the corresponding value of $\mu $ is $2$. From Proposition \ref{class}, we know that when $\mu =2$, then $n=4m^2$, $m$-positive integer. Thus we have $k= \frac{n-\sqrt n}{2} = 2m^2-m$. 
\end{rmk}
In \cite{comb}, a difference set $D$ with parameters $(4m^2; 2m^2 -m;m^2 -m)$ ($m$ a positive integer) is called a {\bf Hadamard difference set}. Thus we have the following corollary. 
\begin{cor}\label{haddif}
Let $G$ be a group of order $n$ and $D$ be a $(n,k,\lambda)$ difference set such that $0\notin D$. Then $D$ forms a signature set for an $(n,k)$-equiangular frame if and only if $D$ is a reversible Hadamard difference set.\end{cor}
From Corollary \ref{haddif}, we infer that the problem of finding signature sets for $(n,\frac{n-\sqrt n}{2})$-equiangular frames is equivalent to finding reversible Hadamard difference sets. In \cite{Dillon}, an explicit construction of a reversible Hadamard difference set in $\bb Z_{2^{a+1}}^2$ for all $a \in \bb N$ is presented. Thus by Corollary \ref{haddif}, there exists $(n,\frac{n-\sqrt n}{2})$-equiangular frames for all $n$ of the type $2^{2a}$ where $a\in \bb N$. Following is an example of such a difference set in \cite{comb}.
\begin{eg} Let $G= \bb Z_8 \times Z_8$. Then the set $D= A\cup A^{-1}$ where \[ A = \{ab^4,ab^5,ab^6,ab^7,a^2b^2,a^2b^3,a^2b^6,a^2b^7,a^3b^2,a^3b^4,a^3b^5,a^3b^7,a^4b,a^4b^3\}\]is a $(64,28,12)$ reversible Hadamard difference set. Since $e\notin D$, from Corollary \ref{haddif}, $D$ forms a signature set for a $(64,28)$-equiangular frame.
 \end{eg}
\vspace{0.05in}
\begin{rmk}
The sets studied in Proposition \ref{propeg1} and Proposition \ref{propeg2} are reversible Hadamard difference sets not containing the identity. This is another way to see that these sets form signature sets for $(16,6)$ and $(36,15)$-equiangular frames respectively.
\end{rmk}

The following proposition gives us a relation between reversible Hadamard difference sets and Hadamard matrices. 
\begin{prop}Let $D$ be a $(n, k,\lambda)$ reversible Hadamard difference set with $0\notin D$ and let $Q = \sum_{g\in D}\lambda(g)-\sum_{g\in T}\lambda(g)$, where $T= D^c\setminus\{0\}$. Then the matrix $H =I-Q$ is a Hadamard matrix.
\end{prop}
\begin{proof}Since $D$ is a $(n,k,\lambda)$ reversible Hadamard difference set with $0\notin D$, by Corollary \ref{haddif}, $D$ forms a signature set for an $(n,k)$-equiangular frame. Thus $Q =\sum_{g\in D}\lambda(g)-\sum_{g\in T}\lambda(g) $ is a signature matrix for the $(n,k)$-equiangular frame. Let $H=I-Q$. Then for all $i,j$, $h_{i,j}=\pm 1$. Consider \begin{align*}H^2& = (I-Q)^2\\ & = I-2Q+Q^2\\ & = I -2Q+ (n-1)I + 2Q && \text{by Theorem \ref{sigm}}\\ &= nI.\end{align*} Also $Q= Q^{*}$ implies $H=H^{*}$. Thus we have $HH^{*}= nI$ and hence $H=I-Q$ is a Hadamard matrix.
\end{proof}
\section{Quasi-Signature Sets and Real Equiangular Frames}\label{sec4}
If $Q$ is a Seidel matrix, we say that $Q$ is in a standard form if its first row and
column contains only $1$'s except on the diagonal, as shown below: 
\[Q=\left[\begin{array}{ccccc}
0&1&\hdots&\hdots& 1\\
1&0&*&\hdots&*\\
\vdots&*&\ddots&\ddots&\vdots\\
\vdots&\vdots&\ddots&\ddots&\vdots\\
1& * &\hdots &\hdots &0 \end{array}\right]\]
 We say that it is trivial if it has a standard form which has all of its off-diagonal entries
equal to $1$ and nontrivial if at least one off-diagonal entry is not equal to $1$. 
One can verify by conjugation with an appropriate diagonal unitary that the
equivalence class of any Seidel matrix contains a matrix of standard form. 
In the real case, the off-diagonal entries of $Q$ are in the set $\{-1, 1\}$ and in the complex case the off-diagonal entries of $Q$ are roots of unity as shown in \cite{cuberoots} where off-diagonal entries are cube roots of unity.\\
\\Let $G$ be a group of order $m$. Let $\lambda: G \longrightarrow GL(\bb F(G))$ be the left regular representation such that 
$\lambda(g)e_h=e_{gh}$. Then we know that $\sum_{g\in G}\lambda(g)= J$ where $J$ is the $m\times m$ matrix of all $1's$. As in Section \ref{sec2}, we will form signature matrices using the left regular representation of a group, but in the standard form. \\
Lets observe the following definition in analogy with the Definition \ref{defsigset} for constructing signature matrices in the standard form.
\vspace{0.022in}
\begin{defn}\label{defqsigset}
Let $S \subset G\setminus\{e\}$ and $T = S^c\setminus \{e\}$ such that $G\setminus \{e\} = S\cup T$. As before let $ Q = \sum_{g\in S}\lambda(g)- \sum_{h\in T}\lambda(h)$. Form \[\tilde Q = \left[\begin{array}{c|c}
0&C^t\\
\hline
C&Q\\
\end{array}\right]\] where \[C = \left(\begin{array}{c}1\\\vdots \\1\end{array}\right)\in \bb C^m\] Then $\tilde Q$ is an $(m+1)\times (m+1)$ matrix with $\tilde Q_{ii} = 0$ and $|\tilde Q_{ij} | = 1$ for all $i \neq j$. Let $n=m+1$, then we call $S$ a quasi-signature set for an $(n,k)$ equiangular frame if $\tilde Q$ forms a signature matrix for an $(n,k)$ equiangular frame. 
\end{defn}
\vspace{0.03in}
Following is a necessary and sufficient condition for a set in a group $G$ to be a quasi-signature set for an $(n,k)$-equiangular frame.
\vspace{0.02in}
\begin{thm}\label{mainthmqs1}Let $G$ be a group of order $m$ and $S\subset G\setminus\{e\}$, $T = S^c\setminus\{e\}$ such that $G\setminus\{e\}= S\cup T$. Then there exists a $k$ such that $S$ forms a quasi-signature set for an $(n,k)$-equiangular frame, where $n=m+1$, if and only if the following hold: 
\begin{enumerate}
\item[(a)]$g\in S$ implies $g^{-1}\in S$ and $h\in T$ implies $h^{-1}\in T$;
\item[(b)]\label{condq12} for all $g \in S$, $N_{(S,S)}^g - 2N_{(S,T)}^g +N_{(T,T)}^g  = \mu -1$ and for all $h\in T$, $N_{(S,S)}^h - 2N_{(S,T)}^h  +N_{(T,T)}^h  = -\mu -1$ where $\mu=|S|-|T|$.\end{enumerate}
$k$ is related to $\mu$ by the equations given in \eqref{relation}.
\end{thm}
\begin{proof} Form $Q= \sum_{g\in S}\lambda(g)- \sum_{h\in T}\lambda(h)$ and \[\tilde Q = \left[\begin{array}{c c}
0&C^t\\
C&Q\\
\end{array}\right]\] where \[C = \left(\begin{array}{c}1\\\vdots \\1\end{array}\right)\in \bb C^m\]  Then by Definition \ref{defqsigset}, $S$ will form a quasi-signature set for an $(n,k)$-equiangular frame if and only if $\tilde Q$ forms a signature matrix for an $(n,k)$-equiangular frame. From Theorem \ref{sigm}, $\tilde Q$ will form a signature matrix for an $(n,k)$ equiangular frame if and only if it satisfies the following two conditions:
\begin{enumerate}
\item[(a)] $\tilde Q$ is self adjoint that is $\tilde Q = \tilde Q^*$; and 
\item[(b)] $\tilde Q^2 = (n-1)I + \mu \tilde Q$ for some real number $\mu$.
\end{enumerate}
The condition $\tilde Q = \tilde Q^*$ is equivalent $Q=Q^*$ which is equivalent to \begin{align*}\sum_{g\in S}\lambda(g)- \sum_{h\in T}\lambda(h)& = (\sum_{g\in S}\lambda(g)- \sum_{h\in T}\lambda(h))^*\\
& = \sum_{g\in S}\lambda(g^{-1})- \sum_{h\in T}\lambda(h^{-1}).\end{align*} 
Thus $g\in S$ implies $g^{-1}\in S$ and $h\in T$ implies $h^{-1} \in T$.\\
For the second condition we need $\tilde Q^2=(n-1)I +\mu \tilde Q $. We have  
 \[\tilde Q^2 = \left[\begin{array}{cc}
n-1& \tilde C^t  \\
\tilde C & J +  Q^2  
\end{array}\right]\] where $\tilde C=(|S|-|T|)C$ and $J$ is an $m\times m$ matrix of all $1$'s.
Thus $\tilde Q^2= (n-1)I +\mu \tilde Q$ if and only if 
\begin{enumerate}
\item[(a)]$|S|-|T|=\mu$; and 
\item[(b)]$J +  Q^2 = (n-1)I + \mu Q $ that is $Q^2 = (n-1)I + \mu Q - J$. \\Since $J = \sum_{g\in G}\lambda (g)$, we have \begin{align*}  Q^2 & = (n-1)I + \mu Q -J\\
& = (n-2)I + \mu(\sum_{g\in S}\lambda(g)- \sum_{h\in T}\lambda(h)) - \sum_{g\in G\setminus\{e\}}\lambda (g) \\
& = (n-2)I + (\mu-1)\sum_{g\in S}\lambda(g)- (\mu+1)\sum_{h\in T}\lambda(h).\end{align*}
\end{enumerate}
By same counting arguments as used before in Theorem \ref{mainthm1}, we get $Q^2=(n-2)I + (\mu-1)\sum_{g\in S}\lambda(g)- (\mu+1)\sum_{h\in T}\lambda(h)$ if and only if for all $g\in S$, $N_{(S,S)}^g - 2N_{(S,T)}^g  +N_{(T,T)}^g  = \mu -1$ and 
for all $h\in T$, $N_{(S,S)}^h - 2N_{(S,T)}^h  +N_{(T,T)}^h  = -\mu -1$ where $\mu=|S|-|T|$.
\end{proof}
\vspace{0.02in}
\begin{rmk}
 From Theorem \ref{mainthmqs1}, note that if $S\subset G$ is a quasi-signature set for an $(n,k(\mu))$-equiangular frame then $|S|-|T|=\mu$.
\end{rmk}
\vspace{0.02in}
\begin{prop}\label{card}Let $G$ be a group of order $m$ and $S\subset G\setminus\{e\}$, $T=S^c\setminus \{e\}$. Then the condition $|S|-|T|=\mu$ for some integer $\mu$ is equivalent to $|S| = \frac{n-2+\mu}{2}$ and $|T|= \frac{n-2-\mu}{2}$ where $n=m+1$.
\end{prop}
\begin{proof}Let $\mu \in \bb Z$. If $|S| = \frac{n-2+\mu}{2}$ and $|T|= \frac{n-2-\mu}{2}$, then $|S|-|T|=\mu$. \\
If $|S|-|T|=\mu$, then since $G\setminus \{e\}= S\cup T$, we have $|S|+|T| = m-1=n-2$. Thus solving for $|S|$ and $|T|$, we get that $|S| = \frac{n-2+\mu}{2}$ and $|T|= \frac{n-2-\mu}{2}$.
\end{proof}
\vspace{0.02in}
\begin{eg}We know that there exists a $(6,3)$-equiangular frame. In this example we will show that this frame comes from a quasi-signature set $S$ in $\bb Z_5$. Let $G = (\bb Z_5,+)$. Since $\mu = 0$, we must have $|S| = |T| = 2$. If we take $S = \{1,4\}$ and $T= \{2,3\}$, then we form $Q = \sum_{g\in S}\lambda(g)-\sum_{h\in T}\lambda(h)$, then $Q$ is given from the following multiplication table  
\begin{center}
\begin{tabular}{l|lllll}
$+$ & 0 & 1 & 2 & 3 & 4 \\ 
\hline
0 & 0 & 1 & -1 & -1 & 1 \\ 
4 & 1 & 0 & 1 & -1 & -1 \\ 
3 & -1 & 1 & 0 & 1 & -1 \\ 
2 & -1 & -1 & 1 & 0 & 1 \\ 
1 & 1 & -1 & -1 & 1 & 0
\end{tabular}
\end{center}
that is \[Q = \left[\begin{array}{ccccc}
 0 & 1 & -1 & -1 & 1 \\ 
 1 & 0 & 1 & -1 & -1 \\ 
 -1 & 1 & 0 & 1 & -1 \\ 
 -1 & -1 & 1 & 0 & 1 \\ 
 1 & -1 & -1 & 1 & 0
\end{array}\right]\] 
Thus we can form $\tilde Q$ as follows 
 \[\tilde Q = \left[\begin{array}{ccc}
0&C^t\\
C&Q \end{array}\right]= \left[\begin{array}{cccccc}
0&1&1&1&1&1\\
1& 0 & 1 & -1 & -1 & 1 \\ 
 1&1 & 0 & 1 & -1 & -1 \\ 
 1&-1 & 1 & 0 & 1 & -1 \\ 
 1&-1 & -1 & 1 & 0 & 1 \\ 
 1& 1 & -1 & -1 & 1 & 0
\end{array}\right]\] 
It can be easily verified that $\tilde Q = \tilde Q^*$ and $\tilde Q^2 = 5I$. Thus $S$ forms a quasi-signature set for $(6,3)$-equiangular frame.
\end{eg} 
Again as before in Section \ref{sec2}, we can simplify the conditions given in Theorem \ref{mainthmqs1} as follows:
\begin{thm}\label{mainthmqs2}Let $G$ be a group with $|G|= m$ and $S,T\subset G\setminus\{e\}$ be disjoint such that $G\setminus \{e\}= S \cup T$. Also let $S = S^{-1}$ and $T = T^{-1}$. Then there exists a $k$ such that $S$ forms a quasi-signature set for an $(n,k)$-equiangular frame where $n=m+1$ if and only if for all $g\in S$, \begin{equation}\label{condqs3}N_{(S,S)}^g = \frac {n+3\mu-6}{4}\end{equation} and for all $h\in T$, \begin{equation}\label{condqs4} N_{(T,T)}^h = \frac {n-3\mu-6}{4},\end{equation} where $\mu = |S|-|T|$. Here $k$ and $\mu$ are related by Equation \eqref{relation}.
 \end{thm}
\begin{proof} Let us assume that $|S|-|T|=\mu$, for some $\mu \in \bb Z$. Using Proposition \ref{card}, we have $|S| = \frac{n-2+\mu}{2}$. 
For $g \in S$, let $N_{(S,S)}^g =m_1$. Since $|S|-1= N_{(S,S)}^g+ N_{(S,T)}^g$, we have $N_{(S,T)}^g = |S|-1-m_1$. Also, we have $N_{(S,T)}^g+N_{(T,T)}^g = |T| = |G|-1-|S|$. Thus \begin{align*}N_{(T,T)}^g& =|G|-1-|S|-N_{(S,T)}^g \\& = n-2-|S|-(|S|-1-m_1)\\ & = n-2|S|+m_1-1.\end{align*} 
From Theorem \ref{mainthmqs1}, $S$ will form a quasi-signature set for an $(n,k)$-equiangular frame if and only if condition \eqref{condq12} holds. \\Thus $N_{(S,S)}^g - 2N_{(S,T)}^g  +N_{(T,T)}^g  = \mu-1$ is equivalent to \begin{align*}\mu -1 & = m_1 -2(|S|-1-m_1)+n-2|S|+m_1-1\\&= 4m_1 - 4|S|+1+n\end{align*} or equivalently, we have $n+2-\mu = 4(|S|-m_1)$.\\ 
Thus \begin{align*}m_1&= |S|-\frac{n+2-\mu}{4}\\
      & = \frac{n+3\mu-6}{4}
     .\end{align*}

Now if we let $h\in T$ and $N_{(T,T)}^h= m_2 $, then as above $N_{(S,T)}^h = |T|-1-m_2$. Also, $N_{(S,T)}^h+N_{(S,S)}^h = |S|= |G|-1-|T|$. 
Thus \begin{align*}N_{(S, S)}^h & = n-2-|T|-N_{(S,T)}^h\\
& = n-2-|T|-(|T|-1-m_2) \\
& = n-2|T|+m_2-1.\end{align*}
Again using Theorem \ref{mainthmqs1}, condition \eqref{condq12}, we have $ -N_{(S,S)}^h + 2N_{(S,T)}^h  -N_{(T,T)}^h = \mu +1$ is equivalent to \begin{align*}\mu +1 & = -(n-2|T|+m_2-1) + 2(|T|-1-m_2)-m_2\\ & = -4m_2 +4|T| -n -1 .\end{align*} or equivalently, we have \begin{align*} m_2&= |T|-\frac{\mu+n+2}{4}\\
&= \frac{n-3\mu-6}{4}.\end{align*}
Thus $S$ forms a quasi-signature set for an $(n,k)$-equiangular frame if and only if for all $g\in S$, we have \[N_{(S,S)}^g=\frac{n+3\mu-6}{4}\] and for all $h\in T$, we have \[N_{(T,T)}^h=\frac{n-3\mu-6}{4}.\]
\end{proof}
\begin{rmk}Since $|S|-1= N_{(S,S)}^g+ N_{(S,T)}^g$, we can find $N_{(S,T)}^g$ in terms of $n$ and $\mu$. Thus, \begin{align*}N_{(S,T)}^g&= |S|-1-N_{(S,S)}^g\\
& = \frac{n-2-\mu}{4}
.\end{align*}
 Similarly for all $h \in T$, we have
 \begin{align*}N_{(S,T)}^h&= |T|-1-N_{(T,T)}^h\\
& = \frac{n-2+\mu}{4}.\end{align*}
\end{rmk}
\vspace{0.09in}
\begin{cor}\label{corqs2}
 Let $G$ be a group of order $m$. If there exists a quasi-signature set $S$ in $G$ for an $(n,k(\mu))$-equiangular frame where $n=m+1$, then 
\begin{enumerate}
\item[(a)] $\mu \equiv 0\pmod2$;
\item[(b)] $n\equiv 0\pmod2$;
\item[(c)] \label{rel} $n$, $\mu$ satisfies $2-\frac{n}{3}\leq \mu \leq \frac{n}{3}-2$.\end{enumerate}
\end{cor}
\begin{proof}If there exists a quasi-signature set $S$ in $G$ for an $(n,k(\mu))$-equiangular frame, then we know from Theorem \ref{mainthmqs2}, conditions \ref{condqs3} and \ref{condqs4} hold. Adding \ref{condqs3} and \ref{condqs4}, we get $n\equiv 0\pmod 2$. Subtracting \ref{condqs3} from \ref{condqs4}, we get $\mu \equiv 0\pmod2$.\\ Also $0\leq N_{(S,S)}^g = \frac{3\mu +n-6}{4}$ implies that $2-\frac{n}{3}\leq \mu$ and $0\leq N_{(T,T)}^h =  \frac{-3\mu +n-6}{4}$ implies that $\mu \leq \frac{n}{3}-2$. Thus we have, \begin{equation*} 2-\frac{n}{3}\leq \mu \leq \frac{n}{3}-2.\end{equation*}\end{proof}
Note that in the case of quasi-signature sets, we have a better bound on the value of $\mu$ as compared to the case of signature sets. Thus we have the following proposition that eliminates some of the cases in which $S \subset G$ can be a quasi-signature set for an $(n,k(\mu))$-equiangular frame.  
 
\begin{prop}
 Let $G$ be a group of order $m$ and $n=m+1$, then the following hold.
\begin{enumerate}
\item[(a)] If there exists a quasi-signature set for an $(n,\frac{n}{2})$-equiangular frame, then $n=2a$ where $a\in \bb N$ is an odd number. 
\item[(b)] For an odd prime $p$, if there exists a quasi-signature set for a $(2p,k)$-equiangular frame, then $k=p$.
\item[(c)] For an odd prime $p$, there does not exist a quasi-signature set for a $(4p,k)$-equiangular frame for any value of $k$.\end{enumerate}
\end{prop}
\begin{proof}
If there exists a quasi-signature set $S$ for an $(n,\frac{n}{2})$-equiangular frame, then using Equation \eqref{relation}, $\mu =0$. From Theorem \ref{mainthmqs2}, we have for all $g\in S$ and for all $h\in T$, $N_{(S,S)}^g= N_{(T,T)}^h = \frac{n-6}{4}$. Thus $n\equiv 6\pmod 4$ or equivalently $n\equiv 2\pmod 4$. Thus $n$ is of the form $n=2a$ where $a\in \bb N$ is an odd number.\\ 
 For the second part assume that $S$ forms a quasi-signature set for an $(2p,k)$-equiangular frame where $p$ is an odd prime. Using Corollary \ref{corqs2}, let $\mu=2\mu_1$ for some $\mu_1 \in \bb Z$. Then from Equation \eqref{relation}, we have
\begin{equation}\label{ineq}
 k = \frac{2p}{2}-\frac{2\mu_1\cdotp 2p}{2\sqrt{4(2p-1)+4\mu_1^2}} = p-\frac{p\mu_1}{\sqrt{(2p-1)+\mu_1^2}}
.\end{equation}Thus $p^2\mu_1^2\equiv 0\pmod {(2p-1+\mu_1^2)}$.  If $\mu_1 =0$, then from part (a), we have that $S$ forms a signature set for an $(2p,p)$-equiangular frame. Clearly $\mu_1 \neq \pm 1$ because in that case $p\equiv 0\pmod 2$ which contradicts that $p$ is a prime. If $\mu_1 \neq 0$, then  $p^2\mu_1^2\equiv 0\pmod {(2p-1+\mu_1^2)}$ implies $\mu_1^2-1\equiv 0\pmod p$.  Thus either $\mu_1-1\equiv 0\pmod p$ or $\mu_1+1\equiv 0\pmod p$. But this contradicts \eqref{rel} in Corollary \ref{corqs2} as  $2-\frac{2p}{3}\leq \mu \leq \frac{2p}{3}-2$ implies $1-\frac{p}{3}\leq \mu_1 \leq \frac{p}{3}-1$. Thus $S$ forms a quasi-signature set for a $(2p,p)$-equiangular frame.  \\
Similarly if $S$ forms a quasi-signature set for $(4p,k)$-equiangular frame where $p$ is a prime, then once again using Equation \eqref{relation}, we get 
\begin{equation}\label{ineq2}
 k = \frac{4p}{2}-\frac{2\mu_1\cdotp 4p}{2\sqrt{4(4p-1)+4\mu_1^2}} = 2p-\frac{2p\mu_1}{\sqrt{(4p-1)+\mu_1^2}}
.\end{equation}
From part (a), $\mu \neq 0$. Using the same argument as discussed above $\mu \neq \pm 1$. Thus $4p^2\mu_1^2 \equiv 0 \pmod{(4p-1+\mu_1^2)}$ implies that $\mu_1^2-1\equiv 0\pmod {4p}$. Thus both $\mu_1-1$ and $\mu_1+1$ are even integers. Let $\mu_1+1 = 2a$ for some $a\in \bb Z$, $a\neq 0$. Thus $\mu_1^2-1\equiv 0\pmod {4p}$ implies that $a(a-1)\equiv 0\pmod p$. But this contradicts \eqref{rel} in Corollary \ref{corqs2} as $2-\frac{4p}{3}\leq 2(2a-1)\leq \frac{4p}{3}-2$ is equivalent to $1-\frac{p}{3}\leq a\leq \frac{p}{3}-2$. 
Hence in the case of $n=4p$, there does not exist a quasi-signature set $S$ for an $(n,k)$-equiangular frame for any value of $k$.
\end{proof}
$ $

As in Section \ref{sec2}, we will consider group $G$ of the form $G = C_N \times C_N$, the direct product of groups of order $N$. We will be constructing different equiangular frames by taking subsets in $G$ as before. But now these subsets will act as quasi-signature sets instead. The following two propositions illustrate the type of equiangular frames we get when $G= C_N\times C_N$.
$ $

\begin{prop}Let $G = C_N \times C_N = \langle a,b:a^N=e, b^N=e, ab=ba\rangle$ and let $S = \{a,a^2,\hdots , a^{N-1},b,b^2, \hdots , b^{N-1}\}$. Then $S$ will form a quasi-signature set for an $(n,k)$-equiangular frame where $n= N^2+1$ if and only if $N=3$ and $k=5$.
\end{prop}
\begin{proof} Since $|S|= 2(N-1)$ and $|T|= (N-1)^2$ we have $ \mu = |S|-|T| = -N^2+4N-3$. For all $g\in S$, we have $N_{(S,S)}^g = N-2$ and for all $h\in T$, we have $N_{(S,S)}^h = 2$. Thus $N_{(S,T)}^h= 2N-2-2 =2N-4$ and hence \begin{align*}N_{(T,T)}^h & = |T|-1-N_{(S,T)}^h\\ &  = (N-1)^2 -1-2N-4\\ & = (N-2)^2-1-2N+4\\ &= (N-2)^2.\end{align*} By Theorem \ref{mainthmqs2}, we have $S$ will form a quasi-signature set for an $(n,k(\mu))$-equiangular frame if and only if Equations \ref{condqs3} and \ref{condqs4} hold. Thus we have \[N-2=\frac{N^2-5+3\mu}{4}\] which gives us \[3\mu =4N-N^2-3\] and \[(N-2)^2=\frac {N^2-5 -3\mu}{4}\] which gives us \[3\mu = - 3N^2+16N-21.\]
Solving for $N$ we get $N=3$ and hence we get $\mu =0$. Thus we get $(10,5)$ equiangular frame.
\end{proof}
\begin{prop}Let $G = C_N \times C_N = \langle a,b:a^N=e, b^N=e, ab=ba\rangle$ and let $S = \{a,\hdots , a^{N-1},b, \hdots , b^{N-1},ab,\hdots, a^{n-1}b^{n-1}\}$, then $S$ will form a quasi-signature set for an $(n,k(\mu))$-equiangular frame where $n= N^2+1$ if and only if $N=5$ and $k=13$. 
\end{prop}
\begin{proof} Since $|S|= 3(N-1)$ and $|T|= (N-1)(N-2)$ we have $ \mu = |S|-|T| = -N^2+6N-5$. For all $g\in S$, we have $N_{(S,S)}^g = N$ and for all $h\in S$ we have $N_{(S,S)}^h = 6$. Thus $N_{(S,T)}^h = 3(N-1)-6 = 3N-9$ and \begin{align*}N_{(T,T)}^h & = |T|-1-N_{(S,T)}^h\\ &  = (N-1)^2 -1-2N-4\\ & = (N-2)^2-1-2N+4\\ &= (N-2)^2.\end{align*} By Theorem \ref{mainthmqs2}, we have $S$ will form a quasi-signature set for an $(n,k(\mu))$-equiangular frame if and only if Equations \ref{condqs3} and \ref{condqs4} hold. Thus we have \[N=\frac{N^2-5+3\mu}{4}\] which gives us \[3\mu = -N^2+4N+5\]  and \[N^2-6N+10=\frac {N^2-5 -3\mu}{4}\] which gives us \[3\mu = -3N^2+24N-45.\]
Solving for $N$, we get $-3N^2+24N-45 = -N^2+4N+5$ which implies $2N^2-20N+50=0$ that is $N=5$ and $\mu =0$. Thus we get $(26,13)$ equiangular frame.
\end{proof}
$ $

Note that from the above two propositions, we are getting equiangular frames of the type $(n,\frac{n}{2})$. In \cite{conway}, a real $n \times n$ matrix $C$ with $c_{i,i} = 0$ and $c_{i,j} = \pm 1$ for $i \neq j$ is called a {\bf conference matrix} provided $C^{*}C = (n - 1)I$. It has been shown in \cite{holmes} that every symmetric conference matrix is a signature matrix with $\mu = 0$ and $k = \frac{n}{2}$. There are sufficient number of examples in the literature, see \cite{bod-paulsen} and \cite{cas-eq}, of the equiangular frames of type $(p+1, \frac{p+1}{2})$ where $p$ is a prime. In the next two results, we will characterize some of the equiangular frames of the type $(p+1, \frac{p+1}{2})$ where $p$ is a prime that arise from a quasi-signature set in the group $(\bb Z_p,+)$.

$ $

\begin{thm}\label{mainthmqs3} Let $G= (\bb Z_p,+)$ where $p$-prime. If $(\bb Z_P,\cdot) = \langle 2\rangle$, then the subgroup $\langle 2^2\rangle$ of $(\bb Z_p,\cdot)$ forms a quasi-signature set for a $(p+1, k)$-equiangular frame if and only if $p\equiv 5\pmod 8$. In this case $k=\frac{p+1}{2}$.
\end{thm}
\begin{proof} Let us denote $S = \langle 2^2\rangle$. Using Fermat's Little theorem, we know that $2^{p-1}\equiv 1\pmod p$. Thus, $S = \langle 2^2\rangle = \{2^{2k}:k=1,\hdots, \frac{p-1}{2}\}$. Also, $G=\langle 2\rangle $ implies that $2\notin S$. Since $S$ is a subgroup of $(\bb Z,\cdot)$ of index $2$ and $2\notin S$, we have $(\bb Z,\cdot) = S\cup 2\cdotp S$. Thus $T = 2\cdot S$ and we have $|S|=|T|$. Hence $\mu= |S|-|T|=0$. We will prove this theorem in two steps. First we will show that $S = S^{-1}$ in $(\bb Z_p, +)$.  Secondly we will verify the conditions of Theorem \ref{mainthmqs2}. \\
 Let $\tilde g\in S$, then $\tilde g$ is of the form $2^{2m}$ for some $m \in \{1,\dots, \frac{p-1}{2}\}$. Since $2^{\frac{p-1}{2}}\equiv (p-1)\pmod p$, we have \begin{align*}2^{2m}+2^{\frac{p-1}{2}+2m} & = 2^{2m}(1+2^{\frac{p-1}{2}})\\
& \equiv 0\pmod p.\end{align*}
Thus, $2^{\frac{p-1}{2}+2m} \pmod p$ is the inverse of $\tilde g$ in $(\bb Z_p, +)$. But $2^{\frac{p-1}{2}+2m} \in S$ if and only if $p-1\equiv 0\pmod 4$. Thus, $S$ is closed under inverses with respect to $(\bb Z_p, +)$ if and only if $(p-1) \equiv 0\pmod 4$. Also note that if $\frac{p-1}{4}+m>\frac{p-1}{2}$ that is $2(\frac{p-1}{4}+m)>p-1$, then $2(\frac{p-1}{4}+m)= p-1+2s$ where $s\in \{1, \dots, \frac{p-1}{2}\}$. By using Fermat's Little theorem again we have $2^{2(\frac{p-1}{4}+m)} \equiv 2^{2s}(\text{mod } p)$.\\
For the second part assume that $N_{(S,S)}^{\tilde g} = N$ for some $N\in \bb N$. 
Then for any $g\in S$, $g$ is of the form $2^{2l}$ for some $l \in \{1, \dots, \frac{p-1}{2}\}$.\\ Let us denote \[S\times_g S = \{(g_1,g_2)\in S\times S:g_1+g_2 =g\}.\] For $i,j\in \{1, \dots, \frac{p-1}{2}\}$, since $ 2^{2i}+ 2^{2j}\equiv 2^{2m}\pmod p$ if and only if $2^{2(l-m+i)}+ 2^{2(l-m+j)} \equiv 2^{2l}(\text{mod } p)$, the map $\phi:S\times_{2^{2m}}S \longrightarrow S\times_{2^{2l}}S$ such that \[ \phi((2^{2i},2^{2j})) = (2^{2(l-m+i)}(\text{mod } p),2^{2(l-m+j)}(\text{mod } p))\] is one to one and onto.
Thus $|S\times_{2^{2m}}S| = |S\times_{2^{2l}}S|$. But $|S\times_{2^{2m}}S|= N_{(S,S)}^{2^{2m}}$. Thus for all $g\in S$, we have $N_{(S,S)}^g = N$.\\ Now if $2^{2m}\in S$ then $2^{2m}+2^{2m}=2\cdot 2^{2m} = 2^{2m+1} \in T$. Thus for $g\in S$, if $(g_i,g_j) \in S\times_gS$, then $(g_j,g_i) \in S\times_gS$. Also for all $i$, $(g_i,g_i) \notin S\times_gS$. Hence $N_{(S,S)}^g =|S\times_gS| = N$ must be an even number and let $N = 2r$ for some $r\in \bb N$.\\ 
Similarly if $h\in T$, then $h = 2^{2\tilde m+1}$ for some $\tilde m \in \{1, \dots, \frac{p-1}{2}\}$. 
For $i,j\in \{1, \dots, \frac{p-1}{2}\}$, $2^{2i+1}+ 2^{2j+1} \equiv 2^{2\tilde m+1}\pmod p $ if and only if $2^{2i}+ 2^{2j} \equiv 2^{2\tilde m}\pmod p$. Thus the map that takes $S\times_{2^{2\tilde m}}S \longrightarrow T\times_{2^{2\tilde m+1}}T$ such that \[(2^{2i},2^{2j}) \longrightarrow (2^{2i+1}(\text{mod } p)
,2^{2j+1}(\text{mod } p))\]
 is one to one and onto. Thus, $N_{(T,T)}^{2^{2\tilde m+1}} = N_{(S,S)}^{2^{2\tilde m}}$. But we have $N_{(S,S)}^g = 2r$ for all $g\in S$. Thus for all $h\in T$, we have $ N_{(T,T)}^h = 2r$.
By Theorem \ref{mainthmqs2}, Conditions \ref{condqs3} and \ref{condqs4} hold if and only if we have $\frac{p+1-6}{4} = 2r$ that is $p=8r+5$ or equivalently $p\equiv 5\pmod 8$. But we know from the first part of the proof that $S=S^{-1}$ if and only if $p\equiv 1\pmod 4$. Since $p\equiv 5(\text{mod } 8)$ implies $p\equiv 1(\text{mod } 4)$, we have that $S$ will form a quasi-signature set for an $(p+1,k)$-equiangular frame if and only if $p\equiv 5(\text{mod } 8)$. Also $\mu =0$ implies that $k=\frac{p+1}{2}$.\end{proof}
\vspace{0.1in}
\begin{rmk}
 Note that if $G=(\bb Z_p,+)$, then $|G|=p$. From Section \ref{sec2}, Corollary \ref{even}, we know that there cannot be any signature set in $G$. But if we look at $G=(\bb Z_p, \cdot)$, then $|G|=p-1$. Since $S$ in Theorem \ref{mainthmqs3} is a subgroup of $(\bb Z_p,\cdot)$ of index $2$, from Theorem \ref{subgp} in Section \ref{sec2}, $S$ forms a signature set for the trivial $(p-1,1)$-equiangular frame.
\end{rmk}
\vspace{0.1in}

\begin{eg}Let $G = (\bb Z_{13},+)$ where $(\bb Z_{13},\cdot)= \langle 2 \rangle$. Then using Theorem \ref{mainthmqs2}, $S = \{2^2, 2^4, 2^6, 2^8,2^{10},2^{12}\}$ will form a quasi-signature set for $(14,7)$-equiangular frame. Thus we have $ S=\{4,3,12, 9, 10, 1\}$ and $ T =\{2,5,6,7,8,11\} $. 
We have the following signature matrix $\tilde Q$ for $(14,7)$ equiangular frame:
\vspace{0.1in}
\[ \tilde Q=
\left[\begin{array}{cccccccccccccccc}
0&1&1&1&1&1&1&1&1&1&1&1&1&1\\
1&0&1&-1&1&1&-1&-1&-1&-1&1&1&-1&1\\
1&1&0&1&-1&1&1&-1&-1&-1&-1&1&1&-1\\
1&-1&1&0&1&-1&1&1&-1&-1&-1&-1&1&1 \\
1&1&-1&1&0&1&-1&1&1&-1&-1&-1&-1&1\\
1&1&1&-1&1&0&1&-1&1&1&-1&-1&-1&-1\\
1&-1&1&1&-1&1&0&1&-1&1&1&-1&-1&-1\\
1&-1&-1&1&1&-1&1&0&1&-1&1&1&-1&-1\\
1&-1&-1&-1&1&1&-1&1&0&1&-1&1&1&-1\\
1&-1&-1&-1&-1&1&1&-1&1&0&1&-1&1&1\\
1&1&-1&-1&-1&-1&1&1&-1&1&0&1&-1&1\\
1&1&1&-1&-1&-1&-1&1&1&-1&1&0&1&-1\\
1&-1&1&1&-1&-1&-1&-1&1&1&-1&1&0&1\\
1&1&-1&1&1&-1&-1&-1&-1&1&1&-1&1&0
\end{array}\right]
\]
\end{eg}

\newpage

Next we will be stating an algorithm to generate equiangular frames of the type $(p+1, \frac{p+1}{2})$ using Theorem \ref{mainthmqs3}. We will be considering primes of the type $p\equiv 5\pmod 8$ and then examine which groups of the type $(\bb Z_p,\cdot)$ are generated by $2$. 

$ $

$ $

\noindent \fbox{\parbox{5.435in}{\underline{\textbf{Algorithm for generating $(2k,k)$-equiangular frames using Theorem \ref{mainthmqs3}}}

\smallskip

\noindent Begin by taking a positive integer $m$.

\smallskip

\smallskip

\smallskip

\noindent \textbf{Step 1:} Check whether $8m+5$ is a prime, call it $p$.

\smallskip

\smallskip

\noindent \textbf{Step 2:} For $p$ obtained in Step 1, evaluate $l= 2^i\pmod p$ for each $i \in \{1,\dots,p-2\}$.

\smallskip

\smallskip

\noindent \textbf{Step 3:} If $l\neq 1$ for all $i\in \{1,\dots, p-2\}$, then we get a $(p+1,\dfrac{p+1}{2})$-equiangular frame and the set $\{2^{2r}: 1\leq r\leq \frac{p-1}{2}\}$ is a quasi-signature set for this frame.}}

$ $

$ $

\begin{table}[ht]
\caption{Equiangular frames obtained using the above algorithm for $m<100$} 
 \centering
\small{\begin{tabular}{|l|l|l|l|}

\hline 
 $m$ & $(n,k)$ & m & $(n,k)$\\
\hline \hline 
 $0$ & $( 6,3 )$ & $48$ & $( 390, 195 )$\\

$1$&  $( 14, 7 )$ &$52$ & $( 422, 211 )$ \\
$3$& $( 30, 15 )$ & $57$   &$ (462, 231 )$\\

$4$& $( 38, 19)$ & $63$ & $ ( 510, 255 )$\\
$6$&  $( 54, 27)$ & $67$  & $( 542, 271 )$  \\

 $7$&  $( 62, 31)$ & $69$ & $ ( 558, 279)$ \\

 $12$&  $( 102, 51 )$ & $ 76 $  &   $( 614, 307 )$\\

 $18$ &  $( 150, 75 )$ & $81$  &  $( 654, 327 )$\\

$21$&  $( 174, 87)$ & $82 $  & $ ( 662, 331)$ \\

$22$&   $( 182, 91)$ &  $84 $ &  $( 678, 339 )$\\

$24 $&  $( 198, 99)$ & $ 87$  & $( 702, 351 )  $ \\

$33$ & $( 270, 135)$ & $88 $ &  $( 710, 355) $\\

$36$ & $( 294, 147)$ & $ 94 $ & $( 758, 379 )  $\\
$39$ & $( 318, 159 )$ & $96$ & (774, 387) \\
$43$ &  $( 350, 175 )$ & 99& (798, 399)\\
$46$ & $(374, 187)$ & &\\
\hline
\end{tabular}}
\label{results1}
\end{table}

\begin{rmk}
 Note that $p\equiv 5\pmod 8$ is equivalent to saying that $p=4q+1$ where $q \in \bb N$ is odd. To obtain equiangular frames using Theorem \ref{mainthmqs3}, we have considered groups of the form $(\bb Z_p,\cdot)$, where $p\equiv 5\pmod 8$, that are generated by $2$ that is for which $2$ is a primitive root (mod p). This relates to {\bf Artin's conjecture} which states that ``Every integer $a$, not equal to $-1$ or to a square, is a primitive root (mod p) of infinitely many primes''. In the nineteenth century, several mathematicians proved (see chapter VII in \cite{prim2} for references) that whenever $p$ is of the form $4q+1$, $q$-odd prime, $2$ is a primitive root (mod p). In addition to odd primes, we are also looking for all odd numbers $q$ such that $p=4q+1$ with $2$ as primitive root (mod p). For example as we can see from Table \ref{results1}, that we have $(38,19)$-equiangular frame where $37 = 4\cdotp 9+1$.\\ Thus if Artin's conjecture is true in the case of $2$, then there are infinitely many frames of the type $(n,\frac{n}{2})$ which have conference matrices as signature matrices and $S=\langle 2^2\rangle $ as quasi-signature sets.
\end{rmk}

$ $

Note from Table \ref{results1} that we are not getting all the equiangular frames of the type $(p+1,\frac{p+1}{2})$. For example we did not get $(18,9)$-equiangular frame. The next result will enable us to construct some more equiangular frames again using the group $(\bb Z_p, +)$.
$ $

\begin{thm}\label{mainthmqs4} Let $p$ be a prime of the form $p\equiv 1(\text{mod } 4)$ and $G = (\bb Z_p, +)$. If $\langle 2 \rangle \subset (\bb Z_p,\cdot)$ is a subgroup of index $2$, then $ \langle 2 \rangle $ will form a quasi-signature set in $G$ for a $(p+1,k)$-equiangular frame if and only if $p \equiv 1(\text{mod } 8)$. In this case $k=\frac{p+1}{2}$.
\end{thm}
\begin{proof} Let us denote $S = \langle 2 \rangle = \{2^k: k=1,2, \hdots, \frac{p-1}{2}\}$. Since $S$ is a subgroup of index $2$ in $(\bb Z_p,\cdot)$, then for $a\notin S$, we have $T=a\cdot S$. Thus, $|S|= |T|$ and let $\mu = |S|-|T|$. As before, we will first prove that $S=S^{-1}$ in $(\bb Z_p,+)$ and then we will verify the conditions of Theorem \ref{mainthmqs2}. \\
Let $\tilde g\in S$, then $\tilde g$ is of the form $2^m$ for some $m \in \{1,2, \hdots, \frac{p-1}{2}\}$. Since $2^{\frac{p-1}{2}}\equiv 1(\text{mod } p)$ and $\langle 2\rangle $ is a subgroup of index $2$ in $(\bb Z_p,\cdot)$, we have $2^{\frac{p-1}{4}}\equiv(p-1)(\text{mod } p)$. Thus \begin{align*}2^m+2^{\frac{p-1}{4}+m} & = 2^m(1+2^{\frac{p-1}{4}})\\
& \equiv 0(\text{mod } p) .\end{align*}
Thus, $2^{\frac{p-1}{4}+m} \pmod p$ is the inverse of $\tilde g$ in $(\bb Z_p, +)$. Also note that if $\frac{p-1}{4}+m>\frac{p-1}{2}$, then $\frac{p-1}{4}+m= \frac{p-1}{2}+s$ where $s\in \{1, \dots, \frac{p-1}{2}\}$. Thus $2^{\frac{p-1}{4}} \equiv 2^s (\text{mod } p) \in S$ where $s\in \{1, \dots, \frac{p-1}{2}\}$. Hence $S$ is closed under inverses in the group $(\bb Z_p, +)$.\\
For the second part assume that $N_{(S,S)}^{\tilde g} = N$ for some $N\in \bb N$. For any $g\in S$, $g$ is of the form $2^l$ for some $l \in \bb N$. \\ Let us denote \[S\times_gS = \{(g_1,g_2)\in S\times S:g_1+g_2 =g\}.\] For $i,j \in \{1, \dots, \frac{p-1}{2}\}$, since $ 2^i+ 2^j \equiv 2^m (\text{mod } p) $ if and only if $ 2^{l-m+i}+ 2^{l-m+j} \equiv 2^l(\text{mod } p)$, the map $\phi:S\times_{2^m}S \longrightarrow S\times_{2^l}S$ such that \[\phi((2^i,2^j)) = (2^{l-m+i}(\text{mod } p),2^{l-m+j}(\text{mod } p))\] is one to one and onto.
Thus $|S\times_{2^m} S| = |S\times_{2^l}S|$. But $|S\times_{2^m}S|= N_{(S,S)}^{2^m}$. Thus, for all $g\in S$, we have $N_{(S,S)}^g = N$.\\ Now if $2^m\in S$ then $2^m+2^m=2\cdot 2^m = 2^{m+1} \in S$. Thus for $g\in S$, if $(g_i,g_j) \in S\times_gS$, then $(g_j,g_i) \in S\times_gS$. Also for all $i$, $(g_i,g_i) \in S\times_gS$. Hence $N_{(S,S)}^g =|S\times_gS| = N$ must be an odd number and let $N = 2r-1$ for some $r\in \bb N$.\\ 
Similarly if $h\in T$, then $h = a\cdotp 2^{\tilde m}$ for some $\tilde m \in \{1, \dots, \frac{p-1}{2}\}$ and $a\notin S$. 
Since $a\cdotp 2^i + a\cdotp 2^j \equiv a\cdotp 2^{\tilde m}\pmod p$ if and only if $2^i + 2^j \equiv 2^{\tilde m}\pmod p$, the map that takes \begin{eqnarray*}S\times_{2^{\tilde m}}S \longrightarrow T\times_{a\cdotp 2^{\tilde m}}T, & (2^i,2^j) \longrightarrow (a\cdotp 2^i(\text{mod } p),a\cdotp 2^j(\text{mod } p))\end{eqnarray*} 
 is one to one and onto. Thus $N_{(T,T)}^{a\cdotp 2^{\tilde m}} = N_{(S,S)}^{2^{\tilde m}}$. But $N_{(S,S)}^g = 2r-1$ for all $g\in S$. Thus for all $h\in T$, we have $ N_{(T,T)}^h = 2r-1$.
By Theorem \ref{mainthmqs2}, Conditions \ref{condqs3} and \ref{condqs4} hold if and only if we have $\frac{p+1-6}{4} = 2r-1$ that is $p=8r+1$ or equivalently $p\equiv 1(\text{mod } 8)$. Also $\mu =0$ implies that $k=\frac{p+1}{2}$.\end{proof}
$ $

Lets look at the example when $p=17$.
\begin{eg}Let $G = (\bb Z_{17},+)$. Then $ \langle 2 \rangle$ is a subgroup of index $2$ in $(\bb Z_{17}, \cdot)$. By Theorem \ref{mainthmqs4}, if we take $S = \langle 2 \rangle$ that is $S=\{2,4,8,16, 15,13,9, 1\}$, then $S$ will form a quasi-signature set for $(18,9)$-equiangular frame.
\end{eg}
$ $
The signature matrix obtained from the signature set $S$ in the above example has been shown in \cite{thesis}.
$ $
\begin{rmk}We know from number theoretic arguments that $2^{\frac{p-1}{2}}\equiv 1\pmod p$ if and only if $p\equiv\pm 1\pmod 8$. Thus, in the algorithm for generating equiangular frames of the type $(p+1,\dfrac{p+1}{2})$ from Theorem \ref{mainthmqs4}, we will take primes of the form $p=8m+1$, $m\in \bb N$ and then will check for which of the groups of the type $(\bb Z_p,\cdot)$, we have $\langle 2\rangle$ as a subgroup of index $2$. Note that $73 = 8\cdotp 9+ 1$ but $\langle 2\rangle$ is not a subgroup of index $2$ in $(\bb Z_{73},\cdot)$ and conversely $\langle 2\rangle$ is a subgroup of index $2$ in $(\bb Z_7,\cdot)$ but $7\not \equiv 1\pmod 8$.\end{rmk}

$ $

$ $

\noindent \fbox{\parbox{5.435in}{\begin{center}\underline{\textbf{Algorithm for generating $(2k,k)$-equiangular frames using Theorem \ref{mainthmqs4}}}\end{center}

\smallskip

\noindent Begin by taking a positive integer $m$.

\smallskip

\smallskip

\noindent \textbf{Step 1:} Check whether $8m+1$ is a prime, if so, call it $p$.

\smallskip

\smallskip

\noindent \textbf{Step 2:} For $p$ obtained in Step 1, evaluate $l= 2^i\pmod p$ for each $i \in \{1,\dots,\frac{p-3}{2}\}$.

\smallskip

\smallskip

\noindent \textbf{Step 3:} If $l\neq 1$ for all $i\in \{1,\dots, \frac{p-3}{2}\}$, then we get a $(p+1,\dfrac{p+1}{2})$-equiangular frame and the set $\{2^r: 1\leq r\leq \frac{p-1}{2}\}$ is a quasi-signature set for this frame.}}

$ $

$ $

\begin{table}[ht]
\caption{Equiangular Frames obtained using the above algorithm for $m<300$}
\centering
\small{\begin{tabular}{|c|c|c|c|}
\hline 
 $m$&$(n,k)$ & $m$&$(n,k)$ \\
\hline \hline 
 $0$ & $( 2, 1)$ & $122$ & $( 978, 489)$\\

 $2$ & $( 18, 9)$& $126$ & $( 1010, 505)$\\

 $5$ & $( 42, 21 )$ & $141$ & $( 1130, 565 )$\\

 $12$ & $( 98, 49)$ & $162$ & $( 1298, 649 )$\\

 $17 $& $( 138, 69 )$& $170$ & $( 1362, 681 )$\\

 $24$ & $( 194, 97 )$& $176$& $( 1410, 705 )$\\

 $39$ & $( 314, 157 )$& $186$ & $( 1490, 745 )$\\

 $50$ & $( 402, 201 )$& $212$ & $( 1698, 849)$\\

 $51 $& $( 410, 205 )$& $234 $& $( 1874, 937)$\\

 $56 $& $( 450, 225 )$& $249$ & $( 1994, 997 )$\\

 $65 $& $( 522, 261 )$& $260 $& $( 2082, 1041)$\\

 $71 $& $( 570, 285 )$& $267$ & $( 2138, 1069)$\\

 $95 $& $( 762, 381 )$& $269$ & $( 2154, 1077)$\\

 $96 $& $( 770, 385 )$& $270$ & $( 2162, 1081)$\\

 $101$ & $( 810, 405 )$& $287$ & $( 2298, 1149)$\\

 $107$ & $( 858, 429 )$&$297$&$( 2378, 1189)$ \\

 $116$ & $( 930, 465 ) $&  &\\

 \hline
\end{tabular}}
\label{results2}
\end{table}
\newpage
\section{Cube Roots of Unity and Signature Pairs of Sets}\label{sec7}
It was shown in \cite{cuberoots}, the nontrivial signature matrices whose off-diagonal entries
are cube roots of unity. Also  in \cite{cuberoots}, a  number of necessary and sufficient
conditions for such a signature matrix of an $(n, k)$-equiangular frame to exist are presented. In this section, we will be extending the results we got in the case of real equiangular frames to the case when $Q$ is a cube root signature matrix.

$ $

Let $\omega = \frac{-1}{2}+ i\frac{\sqrt 3}{2}$. Then the set $\{1, \omega, \omega^2\}$ is the set of cube roots of unity. Note also that $\omega^2 =\bar\omega$ and $1 + \omega + \omega^2 = 0$.
In \cite{cuberoots}, a matrix $Q$ is called a cube root Seidel matrix if it is self-adjoint,
has vanishing diagonal entries, and off-diagonal entries which are all cube roots
of unity. Moreover, if $Q$ has exactly two eigenvalues, then we say that it is the cube root
signature matrix of an equiangular tight frame.
\begin{defn}\label{defcube1}Let $G$ be a group with $|G| = n$ and $\lambda:G\longrightarrow GL(\bb F(G))$ be the left regular representation. Let $S, T\subset G\setminus\{e\}$ be disjoint such that $G\setminus\{e\} = S\cup T\cup V$ where $V = (S\cup T)^c\setminus\{e\}$. Form $Q = \sum_{g\in S}\lambda(g)+ \omega\sum_{g\in T}\lambda(g)+ \omega^2\sum_{g\in V}\lambda(g)$. We call $(S,T)$ a signature pair of sets for a cube root $(n,k)$-equiangular frame if $Q$ forms the cube root signature matrix of a $(n,k)$-equiangular frame.
\end{defn}
We will call frames arising from such signature pairs as $(n,k)$-cube root equiangular frames.

Similar to Theorem \ref{mainthm1} in Section \ref{sec2}, we have the following result in the case of $(n,k)$-cube root equiangular frame: 

\begin{thm}\label{mainthmcube1}Let $G$ be a group with $|G| = n$. Let $S, T \subset G\setminus\{e\}$-disjoint such that $G\setminus\{e\} = S\cup T\cup V$ where $V$ is as in Definition \ref{defcube1}. Then there exists a $k$ such that $(S,T)$ will form a signature pair of sets for an $(n,k)$-cube root equiangular frame if and only if the following hold:
\begin{enumerate}
\item $S= S^{-1}$ and $T^{-1} = V$;
\item there exists an integer $\mu$ such that 
\begin{enumerate}
\item[(a)] for all $g\in S$, we have \begin{equation}\label{cube1}\begin{split}& N_{(S,S)}^g + \omega^2 N_{(T,T)}^g + \omega N_{(V,V)}^g +\omega(N_{(S,T)}^g+N_{(T,S)}^g)  +\omega^2( N_{(S,V)}^g+N_{(V,S)}^g)+\\ & N_{(T,V)}^g+N_{(V,T)}^g = \mu;\end{split}\end{equation}
\item[(b)] for all $h\in T$, we have \begin{equation}\label{cube2}\begin{split}& N_{(S,S)}^h + \omega^2 N_{(T,T)}^h + \omega N_{(V,V)}^h +\omega (N_{(S,T)}^h+ N_{(T,S)}^h)  +\omega^2 (N_{(S,V)}^g+ N_{(V,S)}^g)+ \\ & N_{(T,V)}^h+N_{(V,T)}^h = \omega \mu;\end{split}\end{equation}
\item[(c)] for all $\tilde h \in V$, we have \begin{equation}\label{cube3}\begin{split}& N_{(S,S)}^{\tilde h} + \omega^2 N_{(T,T)}^{\tilde h} + \omega N_{(V,V)}^{\tilde h} +\omega (N_{(S,T)}^{\tilde h}+ N_{(T,S)}^{\tilde h}) +\omega^2( N_{(S,V)}^{\tilde h}+N_{(V,S)}^{\tilde h})+\\ & N_{(T,V)}^{\tilde h}+N_{(V,T)}^{\tilde h} = \omega^2 \mu. \end{split}\end{equation}
\end{enumerate} 
\end{enumerate}
In this case $k$ and $\mu$ are related by the equations given in \eqref{relation}
\end{thm}
\begin{proof}
Form $Q = \sum_{g\in S}\lambda(g)+ \omega\sum_{g\in T}\lambda(g)+ \omega^2\sum_{g\in V}\lambda(g)$. Then by Definition \ref{defcube1}, $(S,T)$ will form a signature set for an $(n,k)$-cube root equiangular frame if and only if $Q$ forms a signature matrix for an $(n,k)$-cube root equiangular frame. From Theorem \ref{sigm} we know that an $n\times n$ matrix $Q$ forms a signature matrix for an $(n,k)$-equiangular frame if and only if it satisfies the following two conditions:
\begin{enumerate}
\item[(a)]$Q$ is self adjoint that is $Q = Q^*$; and 
\item[(b)] $Q^2 = (n-1)I + \mu Q$ for some real number $\mu$.
\end{enumerate}
The condition $Q = Q^*$ is equivalent to \begin{align*}\sum_{g\in S}\lambda(g)+ \omega\sum_{h\in T}\lambda(h)+ \omega^2\sum_{\tilde h\in V}\lambda(\tilde h)& = (\sum_{g\in S}\lambda(g)+ \omega\sum_{g\in T}\lambda(g)+ \omega^2\sum_{g\in V}\lambda(g))^*\\& =\sum_{g\in S}\lambda(g^{-1})+ \omega^2\sum_{h\in T}\lambda(h^{-1})+ \omega\sum_{\tilde h\in V}\lambda(\tilde h^{-1}).\end{align*}
Thus $g\in S$ implies $g^{-1}\in S$ and $h\in T$ implies $h^{-1} \in V$. By using counting arguments as before in Theorem \ref{mainthm1}, the second condition $Q^2 = (n-1)I+ \mu Q$  for some real number $\mu$, is equivalent to \begin{enumerate}
 \item[(a)] for all $g\in S$, we have \begin{equation*}\begin{split}& N_{(S,S)}^g + \omega^2 N_{(T,T)}^g + \omega N_{(V,V)}^g +\omega(N_{(S,T)}^g+N_{(T,S)}^g) +\omega^2( N_{(S,V)}^g+N_{(V,S)}^g)+\\ & N_{(T,V)}^g+N_{(V,T)}^g = \mu; \end{split}
\end{equation*}
\item[(b)] for all $h\in T$, we have \begin{equation*}\begin{split}& N_{(S,S)}^h + \omega^2 N_{(T,T)}^h + \omega N_{(V,V)}^h +\omega (N_{(S,T)}^h+ N_{(T,S)}^h)  +\omega^2 (N_{(S,V)}^g+ N_{(V,S)}^g)+ \\ & N_{(T,V)}^h+N_{(V,T)}^h = \omega \mu;\end{split}\end{equation*}
\item[(c)] for all $\tilde h \in V$, we have \begin{equation*} \begin{split}& N_{(S,S)}^{\tilde h} + \omega^2 N_{(T,T)}^{\tilde h} + \omega N_{(V,V)}^{\tilde h} +\omega (N_{(S,T)}^{\tilde h}+ N_{(T,S)}^{\tilde h}) +\omega^2( N_{(S,V)}^{\tilde h}+N_{(V,S)}^{\tilde h})+\\ & N_{(T,V)}^{\tilde h}+N_{(V,T)}^{\tilde h} = \omega^2 \mu.\end{split}\end{equation*}
\end{enumerate}

\end{proof}

\begin{thm}\label{mainthmcube2}Let $G$ be a group with $|G| = n$ and $S, T \subset G\setminus\{e\}$ -disjoint such that $G\setminus\{e\} = S\cup T\cup V$ where $V = (S\cup T)^c\setminus\{e\}$. Also let $S = S^{-1}$ and $T^{-1}=V$. If $(S,T)$ forms a signature pair of sets for an $(n,k(\mu))$-cube root equiangular frame, then the following hold.
\begin{enumerate}
\item[(a)] For all $g\in S$,\begin{equation}\label{cube4}N_{(S,T)}^g  + N_{(T,S)}^g + N_{(T,T)}^g = \frac{n-2-\mu}{3}.\end{equation}
\item[(b)]For all $h\in T$, \begin{equation}\label{cube5} N_{(V,V)}^h + N_{(S,T)}^h  + N_{(S,V)}^h = \dfrac{\mu+n-1}{3}.\end{equation}
\item[(c)]For all $\tilde h\in V$, \begin{equation}\label{cube6}N_{(T,T)}^{\tilde h} + N_{(S,T)}^{\tilde h}+ N_{(S,V)}^{\tilde h} =\dfrac{\mu + n-1}{3}.\end{equation}
\end{enumerate}
\end{thm}
\begin{proof}From Theorem \ref{mainthmcube1}, since $\mu$ is real, using Equation \eqref{cube1} we have, 
\begin{equation}\label{eq1} N_{(T,T)}^g+  N_{(S,V)}^g + N_{(V,S)}^g  =  N_{(V,V)}^g +N_{(S,T)}^g+N_{(T,S)}^g. .\end{equation}
Assume $|S| = l$ and for $g\in S$, $ N_{(T,T)}^g =m$. Then $|T| =|V| = \frac{n-l-1}{2}$ and $N_{(T,S)}^g+ N_{(T,V)}^g = \frac{n-1-l}{2}-m$.
Thus by \eqref{eq1}, we have \[m + N_{(S,V)}^g + N_{(V,S)}^g = N_{(V,V)}^g+\frac{n-1-l}{2}-m -  N_{(T,V)}^g+ \frac{n-1-l}{2}-m-  N_{(V,T)}^g \]
that is \[3m + N_{(S,V)}^g + N_{(T,V)}^g + N_{(V,S)}^g + N_{(V,T)}^g = N_{(V,V)}^g+ n-1-l .\]
But we also have \[N_{(S,V)}^g + N_{(T,V)}^g + N_{(V,V)}^g = \frac{n-1-l}{2}\] and 
\[N_{(V,S)}^g + N_{(V,T)}^g + N_{(V,V)}^g = \frac{n-1-l}{2}. \] Thus,  \[3m + \frac{n-1-l}{2} - N_{(V,V)}^g + \frac{n-1-l}{2} - N_{(V,V)}^g = N_{(V,V)}^g+ n-1-l. \]
Hence we get $ N_{(V,V)}^g =m$ and for all $g\in  S$, \begin{equation}\label{c1} N_{(T,T)}^g = N_{(V,V)}^g.\end{equation}
 Again using \eqref{cube1}, we have $ N_{(S,V)}^g + N_{(V,S)}^g = N_{(S,T)}^g+N_{(T,S)}^g$. But for all $g  \in S$, $ N_{(S,V)}^g + N_{(S,T)}^g = N_{(V,S)}^g +N_{(T,S)}^g$. Thus for all $g\in S$, we have \begin{eqnarray}\label{c2}N_{(S,V)}^g = N_{(T,S)}^g & and & N_{(V,S)}^g = N_{(S,T)}^g.\end{eqnarray}
Using Equations \eqref{c1} and \eqref{c2}, from \eqref{cube1} we get, 
\begin{align*}\mu & = N_{(S,S)}^g - N_{(T,T)}^g -N_{(S,T)}^g  - N_{(T,S)}^g+ N_{(T,V)}^g+ N_{(V,T)}^g \\
& =  N_{(S,S)}^g - N_{(T,T)}^g -N_{(S,T)}^g  - N_{(T,S)}^g + (|T| - N_{(T,T)}^g- N_{(T,S)}^g)+\\ &\quad( |T| - N_{(T,T)}^g-N_{(S,T)}^g)\\
& = N_{(S,S)}^g - 3N_{(T,T)}^g +2|T| -2(N_{(S,T)}^g  + N_{(T,S)}^g)\\
& = N_{(S,S)}^g - 3N_{(T,T)}^g +2|T| -2(N_{(S,T)}^g  + N_{(S,V)}^g)\\
& = N_{(S,S)}^g- 3N_{(T,T)}^g +2|T| -2(l-1-N_{(S,S)}^g)\\
& = 3N_{(S,S)}^g -3N_{(T,T)}^g +n-1-l -2l+2\\
&= 3N_{(S,S)}^g-3l-3N_{(T,T)}^g +n+1.\end{align*}
Thus, $n+1-\mu = 3(l-N_{(S,S)}^g+N_{(T,T)}^g)$ that is $\frac{n+1-\mu}{3} = (l-N_{(S,S)}^g + N_{(T,T)}^g)$. But we know that $N_{(S,T)}^g  + N_{(S,V)}^g = l-1-N_{(S,S)}^g$. Thus, 
\[N_{(S,T)}^g  + N_{(S,V)}^g + N_{(T,T)}^g = \frac{n-2-\mu}{3}\] or by \eqref{c2}
\[N_{(S,T)}^g  + N_{(T,S)}^g + N_{(T,T)}^g = \frac{n-2-\mu}{3}.\] 
To simplify Condition \eqref{cube2}, assume that for $h\in T$, $N_{(S,S)}^h =m$. Then as before, since $\mu$ is a real number, we must have from Condition \eqref{cube2},
\begin{equation}\label{eq2}N_{(S,S)}^h  + N_{(T,V)}^h + N_{(V,T)}^h = N_{(T,T)}^h  + N_{(S,V)}^h + N_{(V,S)}^h .\end{equation}
Also we have \begin{eqnarray}\label{c4}N_{(S,T)}^h  + N_{(S,V)}^h = l-m & and & N_{(T,S)}^h  + N_{(V,S)}^h = l-m\end{eqnarray}
 Thus \eqref{eq2} changes to $m + N_{(T,V)}^h + N_{(V,T)}^h =  N_{(T,T)}^h +l-m- N_{(S,T)}^h+ l-m- N_{(T,S)}^h$ that is $m + N_{(T,V)}^h + N_{(T,S)}^h+ N_{(V,T)}^h + N_{(S,T)}^h=  N_{(T,T)}^h+2l-2m  $. 
But $N_{(T,V)}^h + N_{(T,S)}^h = |T|-1- N_{(T,T)}^h$ and $N_{(V,T)}^h+ N_{(S,T)}^h= |T|-1- N_{(T,T)}^h$. Thus we have $m + 2(|T|-1- N_{(T,T)}^h)=  N_{(T,T)}^h +2l-2m $ that is $N_{(T,T)}^h = m-l-1+\frac{n}{3}$. Again using \eqref{eq2} we have \[m  + N_{(T,V)}^h + N_{(V,T)}^h = m-l-1+\frac{n}{3} + N_{(S,V)}^h + N_{(V,S)}^h. \]
Thus,\begin{equation}\label{c3} N_{(T,V)}^h -  N_{(S,V)}^h + N_{(V,T)}^h -N_{(V,S)}^h = \frac{n}{3}-l-1.\end{equation}
 
But $N_{(T,V)}^h +  N_{(S,V)}^h = |V|-N_{(V,V)}^h$ and $N_{(V,T)}^h +N_{(V,S)}^h =  |V|-N_{(V,V)}^h$. Thus we have, $ |V|-N_{(V,V)}^h - 2 N_{(S,V)}^h + |V|-N_{(V,V)}^h  -2N_{(V,S)}^h = \frac{n}{3}-l-1$ that is  $ 2|V|-2N_{(V,V)}^h - 2 N_{(S,V)}^h  -2N_{(V,S)}^h = \frac{n}{3}-l-1$. If $N_{(V,V)}^h =\tilde m$, then we have \[2(N_{(S,V)}^h  +N_{(V,S)}^h) = n-l-1-2\tilde m - \frac{n}{3}+l+1. \]
That is \[N_{(S,V)}^h  +N_{(V,S)}^h = \frac{ n-2\tilde m - \frac{n}{3}}{2} =\frac {n-3\tilde m}{3}.\]
Thus using \eqref{c3} we have, \[N_{(T,V)}^h  +N_{(V,T)}^h = \frac{n}{3}-l-1+ \frac{ n-2\tilde m - \frac{n}{3}}{2} = \frac{2n-3l -3-3\tilde m}{3}\]
and using \eqref{c4}, we have \[N_{(T,S)}^h  +N_{(S,T)}^h = 2(l-m) - ( \frac {n-3\tilde m}{3}) = \frac{6l-6m-n+3\tilde m}{3}.\]
Substituting the values obtained above in \eqref{cube2} we have,
\begin{align*}\omega \mu & = m+\omega^2(m-l-1+\frac{n}{3})+\omega\tilde m+\omega (\frac{6l-6m-n+3\tilde m}{3})+ \omega^2(\frac {n-3\tilde m}{3})+ \\
& \quad\frac{2n-3l -3-3\tilde m}{3}\\
& = \frac{3m +2n - 3l -3\tilde m - 3+ \omega^2(3m - 3l -3 +2n -3\tilde m)+\omega (6\tilde m + 6l -6m -n)}{3}\\
& = \frac{\omega(9\tilde m +9l -9m -3n+3)}{3}\\
\mu & = 3\tilde m +3l -3m -n+1.\end{align*} Thus \[\frac{\mu + n-1}{3} = \tilde m +l -m \] that is \[N_{(V,V)}^h + N_{(S,T)}^h+ N_{(S,V)}^h = \frac{\mu + n-1}{3}.\] By symmetry, Condition \eqref{cube3} reduces to for all $\tilde h\in V$, we have \[ N_{(T,T)}^{\tilde h} + N_{(S,T)}^{\tilde h}+ N_{(S,V)}^{\tilde h} =\frac{\mu + n-1}{3}.\]\end{proof}
$ $
We have the following corollary which was also proved in \cite{cuberoots} as Proposition 3.3. Here we will be proving it in a totally different way using signature pair of sets in groups.
\begin{cor}Let $G$ be a group of order $n$. If there exist a signature pair $(S,T)$ associated with an $(n,k)$-cube root equiangular frame, then the following hold.
\begin{enumerate}
\item[(a)]$n \equiv 0\pmod3$.
\item[(b)] $\mu$ is an integer and $\mu \equiv 1\pmod 3$.
\item[(c)] The integer $4(n-1)+\mu^2$ is a perfect square and in addition $4(n-1)+\mu^2\equiv 0\pmod 9$.
\end{enumerate} 
\end{cor}
\begin{proof}Let $G$ be a group of order $n$ and let there exist a signature pair $(S,T)$ associated with an $(n,k)$-cube root equiangular frame, then by Theorem \ref{mainthmcube2}, Equations \eqref{cube4} and \eqref{cube5} hold. Adding Equations \eqref{cube4} and \eqref{cube5}, we get $n \equiv 0\pmod3$. Since $n$ is an integer and from Equation\eqref{cube4},  $\frac{n-2-\mu}{3}$ is an integer, thus $\mu$ is also an integer. Let $n=3m$ for some $m$ in $\bb N$. From Equation \eqref{cube6}, we know that $\frac{\mu+n-1}{3}$ is an integer, say $l$. Thus we have, $\mu = 3(l-m)+1$ that is $\mu = 1\pmod 3$.\\
For the third part, using the relation given in \eqref{relation} between $k$ and $\mu$, we have \[k=\frac{3m}{2}(1-\frac{3\tilde{m}+1}{\sqrt{4(3m-1)+(3\tilde m+1)^2}})\] where $n=3m$ for $m\in \bb N$ and $\mu =3\tilde m +1$ for $\tilde m \in \bb Z$.
Thus $4(3m-1)+(3\tilde m+1)^2$ should be perfect square. But\[4(3m-1)+(3\tilde m+1)^2 = 3(3\tilde m^2+2\tilde m+4m-1)\]
Thus, $4(3m-1)+(3\tilde m+1)^2 \equiv 0\pmod3$. But since $4(3m-1)+(3\tilde m+1)^2$ is a perfect square, we have $4(3m-1)+(3\tilde m+1)^2\equiv 0\pmod9$.
\end{proof}

\begin{eg}
 Let $G= (\bb Z_3,+)$. If we take $S = \{1,2\}$ and $T=\emptyset$, then for $g\in S$, $N_{(S,S)}^g =1$. Since it is the only non zero value, using Equation \eqref{cube1}, we have $\mu =1$. The signature matrix is $Q= \left(\begin{smallmatrix} 0 & 1& 1\\1 & 0 & 1\\1& 1& 0 \end{smallmatrix}\right)$ which gives rise to the trivial $(3,1)$-equiangular frame.\\
 If we take $S={\emptyset}$, $T= \{1\}$ and $V=\{2\}$, then for $h\in T$, $N_{(V,V)}^h$ is the only non zero value and equals $1$. Using Equation \eqref{cube2}, we have $\mu = 1$. In this case we get the following signature matrix $Q= \left(\begin{smallmatrix} 0&\omega&\omega^2\\
\omega^2 & 0 & \omega\\
\omega&\omega^2&0 
\end{smallmatrix}\right)$ which is again a signature matrix for the trivial $(3,1)$-equiangular frame.\end{eg}

 Next we would like to examine the case of non-trivial cube root equiangular frames arising from a signature pair of sets. In the process, we will look at some specific values of $n$ and $\mu$ and use the theory we have so far to investigate the possibility of the existence of non-trivial cube root equiangular frames arising from a signature pair of sets.
\begin{lem}\label{lemsp}
 Let $G$ be a finite abelian group of order $m$ where $m\in \bb N$ is odd. Then for every element $ e\neq g\in G$, there exists a unique $h\neq e$ such that $g=h^2$.
\end{lem}
\begin{proof}Firstly we will show that for $g\in G$ there exists an $h\in G$ such that $g=h^2$. The order of $G$ is odd implies that the order of $g$ is odd. Thus there exists an $r\in \bb N$ such that $g^{2r+1}=e$. Thus $g^{2r+1}\cdot g= e\cdot g$ that is $(g^{r+1})^2= g$. Taking $h=g^{r+1}\in G$, we have $h^2=g$. Now suppose that there exists $e\neq h,e\neq \tilde h \in G$ such that $g=h^2$ and $g=\tilde h^2$. Then $h^2\tilde h^{-2}= e$. Since $G$ is abelian, we have $(h\tilde h^{-1})^2=e$. This is only possible when $h\tilde h^{-1}= e$ that is $h= \tilde h$.
\end{proof}

\begin{prop}\label{nmu}
 Let $G$ be an abelian group of order $n$ where $n\equiv 3 \pmod 6$ and let $\mu \equiv 4\pmod 6$. Then there does not exist a signature pair of sets in $G$ associated with a $(n,k(\mu))$-cube root equiangular frame. 
\end{prop}
\begin{proof}
Suppose on the contrary that there exists a signature pair $(S,T)$ associated with $(n,k(\mu))$-cube root equiangular frame. Then using  \eqref{cube4} in Theorem  \ref{mainthmcube2}, we have $N_{(S,T)}^g  + N_{(T,S)}^g + N_{(T,T)}^g = \frac{n-2-\mu}{3}$. Since $n \equiv 3 \pmod 6$ and $\mu \equiv 4\pmod 6$, there exists $k\in\bb N,$ and $k^{'}\in \bb Z$ such that \[\frac{n-2-\mu}{3}= \frac{6k+4-2-(6k^{'}+4)}{3}= 2(k-k^{'})-1\] Also, $G$ abelian implies $N_{(S,T)}^g  = N_{(T,S)}^g$. Hence for all $g\in S$, we have \begin{equation}\label{odd}2N_{(S,T)}^g + N_{(T,T)}^g = 2(k-k^{'})-1.\end{equation} Since the right hand side of \eqref{odd} is odd, $N_{(T,T)}^g$ must be a positive odd integer. For $\tilde g\in S$, by Lemma \ref{lemsp} we know that there exists a unique $e\neq h\in G$ such that $h^2 = \tilde g$. Since $N_{(T,T)}^{\tilde g}$ is odd, $h$ must be in $T$. Thus $h^{-1} \in V$. Since $S$ is closed under inverses, we have $\tilde g^{-1}\in S$ and $\tilde g^{-1}=h^{-2} = (h^{-1})^2$. This contradicts that $N_{(T,T)}^{\tilde g^{-1}}$ is odd. Thus there does not exist a signature pair of sets in $G$ associated with the $(n,k)(\mu))$-cube root equiangular frame.
\end{proof}

\begin{rmk}
 It was shown in \cite{cuberoots} that there exists a $(9,6)$-cube root equiangular frame with $\mu =-2$. Also we know from Theorem 2.1/Exercise 5.13 in \cite{hungerford} that there are two distinct groups of order $9$ that is $\bb Z_9$ and $\bb Z_3\times \bb Z_3$, both abelian. Since $9\equiv 3\pmod 6$ and $-2\equiv 4\pmod 6 $, using Proposition \ref{nmu}, there does not exist a signature pair of sets in $\bb Z_9$ or $\bb Z_3\times \bb Z_3$ associated with the $(9,6)$-cube root equiangular frame. The next possible $(n,k)$ value of a cube root equiangular frame listed in \cite{cuberoots} is $(33,11)$ with $\mu = 4$. Once again $33\equiv 3\pmod 6$ and $4\equiv 4\pmod 6 $. Using Proposition 6.1 in \cite{hungerford}, we infer that every group of order $33$ is isomorphic to the cyclic group $\bb Z_{33}$. Thus using Proposition \ref{nmu}, we can conclude that there does not exist a signature pair of sets in $\bb Z_{33}$ associated with a $(33,11)$-cube root equiangular frame. 
\end{rmk}
This motivates us to explore the quasi-signature case. Similar to Section \ref{sec4}, where we had real signature matrices in the standard form, next we will look at the cube root signature matrices in the standard form. 
\section{Cube Roots of Unity and quasi-signature Pairs of Sets }
In this section we will  consider the signature matrices in the standard form with entries as cube roots of unity. We know from Lemma 2.2 in \cite{cuberoots} that if $Q$ is an $n\times n$ cube root Seidel matrix, then it is switching equivalent
to a cube root Seidel matrix of the form
\[Q^{'}=\left[\begin{array}{ccccc}
0&1&\hdots&\hdots& 1\\
1&0&*&\hdots&*\\
\vdots&*&\ddots&\ddots&\vdots\\
\vdots&\vdots&\ddots&\ddots&\vdots\\
1& * &\hdots &\hdots &0
\end{array}\right]\]
where the $*$'s are cube roots of unity. Moreover, $Q$ is the signature matrix of an
equiangular $(n, k)$-frame if and only if $Q^{'}$ is the signature matrix of an equiangular
$(n, k)$-frame.\\
Thus we have the following definition when cube root signature matrix is in the standard form:
\begin{defn}\label{defqsp}
Let $G$ be a group such that $|G|=m$. Let $S, T\subset G\setminus\{e\}$ be disjoint such that $G\setminus\{e\} = S\cup T\cup V$ where $V = (S\cup T)^c\setminus\{e\}$. For $\omega = \frac{-1}{2}+ i\frac{\sqrt 3}{2}$, form $ Q = \sum_{g\in S}\lambda(g)+\omega \sum_{h\in T}\lambda(g)+ \omega^2 \sum_{\tilde h\in V}\lambda(\tilde h)$ as in Section \ref{sec7}. Let \[\tilde Q = \left[\begin{array}{c|c}
0&C^t\\
\hline
C&Q\\
\end{array}\right]\]  where \[C = \left(\begin{array}{c}1\\\vdots \\1\end{array}\right)\in \bb C^m\]Then we call $(S,T)$ a quasi-signature pair for an $(n,k)$-cube root equiangular frame where $n=m+1$ if $\tilde Q$ forms a cube root signature matrix for an $(n, k)$-cube root equiangular frame. 
\end{defn}
Analogous to Theorem \ref{mainthmqs1} that gives us a necessary and sufficient condition for the existence of quasi-signature set, we have the following result about the quasi-signature pair of sets:
\begin{thm}\label{mainthmqsp1}Let $G$ be a group such that $|G|= m$. Let $S, T\subset G\setminus\{e\}$ be disjoint such that $G\setminus\{e\} = S\cup T\cup V$ where $V = (S\cup T)^c\setminus\{e\}$. Then there exists a $k$ such that $(S,T)$ will form a quasi-signature pair of sets for an $(n,k)$-cube root equiangular frame if and only if the following hold:
\begin{enumerate}
\item $S=S^{-1}$ and $T^{-1} = V$;
\item\begin{enumerate}
\item[(a)] for all $g\in S$, \begin{equation*}\begin{split}& N_{(S,S)}^g + \omega^2 N_{(T,T)}^g + \omega N_{(V,V)}^g +\omega (N_{(S,T)}^g + N_{(T,S)}^g) +\omega^2( N_{(S,V)}^g +N_{(V,S)}^g)+\\ & N_{(T,V)}^g+ N_{(V,T)}^g = \mu -1;\end{split}\end{equation*}
  \item[(b)] for all $h\in T$, \begin{equation*}\begin{split}&N_{(S,S)}^h + \omega^2 N_{(T,T)}^h + \omega N_{(V,V)}^h +\omega (N_{(S,T)}^h + N_{(T,S)}^h) +\omega^2( N_{(S,V)}^h +N_{(V,S)}^h)+ \\& N_{(T,V)}^h+ N_{(V,T)}^h  = \omega \mu-1;\end{split}\end{equation*}
 \item[(c)] for all $\tilde h\in V$, \begin{equation*}\begin{split}& N_{(S,S)}^{\tilde h} + \omega^2 N_{(T,T)}^{\tilde h} + \omega N_{(V,V)}^{\tilde h}+ \omega (N_{(S,T)}^{\tilde{h}} + N_{(T,S)}^{\tilde{h}}) +\omega^2( N_{(S,V)}^{\tilde{h}} +N_{(V,S)}^{\tilde{h}})+\\&N_{(T,V)}^{\tilde{h}}+ N_{(V,T)}^{\tilde{h}}  = \omega^2 \mu-1,\end{split}\end{equation*}
 \end{enumerate}
\end{enumerate}
where $\mu=|S|-|T|$ and is related to $k$ by equations given in \eqref{relation}.
\end{thm}
\begin{proof}
 Form $Q= \sum_{g\in S}\lambda(g)+\omega \sum_{h\in T}\lambda(h)+ \omega^2 \sum_{\tilde h\in V}\lambda(\tilde h)$ and \[\tilde Q = \left[\begin{array}{c |c}
0&C^t\\
\hline
C&Q\\
\end{array}\right]\] where \[C = \left(\begin{array}{c}1\\\vdots \\1\end{array}\right)\in \bb C^m\]  Then by Definition \ref{defqsp}, $(S,T)$ will form a quasi-signature pair of sets for an $(n,k)$-cube root equiangular frame if and only if $Q$ forms a signature matrix for an $(n,k)$-equiangular frame. From Theorem \ref{sigm}, $\tilde Q$ will form a signature matrix for an $(n,k)$ equiangular frame if and only if it satisfies the following two conditions:
\begin{enumerate}
\item[(a)] $\tilde Q$ is self adjoint that is $\tilde Q = \tilde Q^*$; and 
\item[(b)] $\tilde Q^2 = (n-1)I + \mu \tilde Q$ for some real number $\mu$.
\end{enumerate}
The condition $\tilde Q = \tilde Q^*$ is equivalent to $Q=Q^*$ which is equivalent to saying that $g\in S$ implies $g^{-1}\in S$ and $h\in T$ implies $h^{-1} \in V$.\\
For the second condition we need $\tilde Q^2=(n-1)I +\mu \tilde Q $. We have  
 \[\tilde Q^2 = \left[\begin{array}{cc}
n-1& \tilde C^t  \\
\tilde C & J +  Q^2  
\end{array}\right]\] where $\tilde C=\alpha C$, $\alpha =|S|+\omega|T|+\omega^2|V|= |S|-|T|$.
Thus $\tilde Q^2= (n-1)I +\mu \tilde Q$ if and only if 
\begin{enumerate}
\item[(a)] $\alpha = |S|-|T|=\mu$; and 
\item[(b)] $J +  Q^2 = (n-1)I + \mu Q $ that is $Q^2 = (n-1)I + \mu Q - J$. Since $J = \sum_{g\in G}\lambda (g)$, we have \begin{align*}  Q^2 & = (n-1)I + \mu Q -J\\
& = (n-2)I + \mu(\sum_{g\in S}\lambda(g)+\omega \sum_{h\in T}\lambda(h)+\omega^2 \sum_{\tilde{h}\in T}\lambda(\tilde{h})) - \sum_{g\in G\setminus\{e\}}\lambda (g) \\
& = (n-2)I + (\mu-1)\sum_{g\in S}\lambda(g)+ (\omega \mu-1)\sum_{h\in T}\lambda(h)+(\omega^2 \mu-1)\sum_{\tilde{h}\in T}\lambda(\tilde{h})
.\end{align*}
\end{enumerate}
By the same counting arguments as before we have that $\tilde Q^2= (n-1)I +\mu \tilde Q$ if and only if for all $g\in S$,  \begin{equation*}\begin{split}& N_{(S,S)}^g + \omega^2 N_{(T,T)}^g + \omega N_{(V,V)}^g +\omega (N_{(S,T)}^g + N_{(T,S)}^g) +\omega^2( N_{(S,V)}^g +N_{(V,S)}^g)+\\ & N_{(T,V)}^g+ N_{(V,T)}^g = \mu -1,\end{split}\end{equation*}
 for all $h\in T$, \begin{equation*}\begin{split}&N_{(S,S)}^h + \omega^2 N_{(T,T)}^h + \omega N_{(V,V)}^h +\omega (N_{(S,T)}^h + N_{(T,S)}^h) +\omega^2( N_{(S,V)}^h +N_{(V,S)}^h)+ \\& N_{(T,V)}^h+ N_{(V,T)}^h  = \omega \mu-1,\end{split}\end{equation*}
for all $\tilde h\in V$, \begin{equation*}\begin{split}& N_{(S,S)}^{\tilde h} + \omega^2 N_{(T,T)}^{\tilde h} + \omega N_{(V,V)}^{\tilde h}+ \omega (N_{(S,T)}^{\tilde{h}} + N_{(T,S)}^{\tilde{h}}) +\omega^2( N_{(S,V)}^{\tilde{h}} +N_{(V,S)}^{\tilde{h}})+\\&N_{(T,V)}^{\tilde{h}}+ N_{(V,T)}^{\tilde{h}}  = \omega^2 \mu-1\end{split}\end{equation*} where $\mu = |S|-|T|$.
\end{proof}
\vspace{0.02in}
\begin{rmk}
 From Theorem \ref{mainthmqsp1}, note that if $S,T\subset G$ is a quasi-signature pair of sets for an $(n,k(\mu))$-cube root equiangular frame then $|S|-|T|=\mu$.
\end{rmk}
\vspace{0.02in}
\begin{prop}\label{cardqsp}Let $G$ be a group of order $m$ and $S,T,V\subset G\setminus\{e\}$ be pairwise disjoint such that $G\setminus\{e\} = S\cup T\cup V$ where $S=S^{-1}$ and $V= T^{-1}$. Then the condition $|S|-|T|=\mu$ for some integer $\mu$ is equivalent to \begin{eqnarray}|S| = \frac{n+2\mu -2}{3} & and & |T| = \frac{n-2-\mu}{3} \end{eqnarray} where $n=m+1$.
\end{prop}
\begin{proof}Let $\mu$ be an integer. Since $G\setminus\{e\} = S\cup T\cup V$ where $S,T,V$ are pairwise disjoint with $S=S^{-1}$ and $V= T^{-1}$, we have $|S|+ 2|T|= n-2$. If $|S|-|T|=\mu$, then solving these equations for $|S|$ and $|T|$, we get $|S| = \frac{n+2\mu -2}{3}$ and $|T| = \frac{n-2-\mu}{3}$.\\
Conversely if $|S| = \frac{n+2\mu -2}{3}$ and $|T| = \frac{n-2-\mu}{3}$, then $|S|-|T|=\mu$.
\end{proof}
\vspace{0.02in}
We know from \cite{cuberoots} that there exists a $(9,6)$-cube root equiangular frame. In the following example we will show that certain subsets of the group of quaternions forms a quasi-signature pair for a $(9,6)$-cube root equiangular frame.
\begin{eg}Let $G =\{1, -1, i, -i, j, -j, k, -k\}$ be the group of quaternions where $i^2=-1$, $j^2=-1$, $k^2=-1$, $i\cdotp j=k$, $j\cdotp k=i$, $k\cdotp i=j$, $j\cdotp i = -k$, $k\cdotp j=-i$ and $i\cdotp k=-j$. Using Proposition \ref{cardqsp}, let us take $S$, $T$, $V$ as follows: \begin{eqnarray*}S = \{-1\},  & T = \{i,j,k\}, & V = \{-i,-j,-k\}.\end{eqnarray*} Then all the conditions of the Theorem \ref{mainthmqsp1} are satisfied and we get $\mu = -2$. Using \ref{relation}, we get $k=6$. Hence $(S,T)$ forms a quasi-signature pair for $(9,6)$-cube root equiangular frame. The signature matrix for a $(9,6)$-cube root equiangular frame is given below in the standard form (here $n=|G|+1 =9$): 
\[\tilde{Q}=\begin{bmatrix}
0&1&1 & 1& 1&1&1&1&1\\
1& 0 & 1 & \omega & \omega^2 & \omega & \omega^2 & \omega & \omega^2\\
1& 1 & 0 & \omega^2 & \omega & \omega^2 & \omega & \omega^2 & \omega\\
1&\omega^2 & \omega & 0 & 1 & \omega^2 & \omega & \omega & \omega^2\\
1& \omega & \omega^2 & 1 & 0 & \omega & \omega^2 & \omega^2 & \omega\\ 
1&\omega^2 & \omega & \omega & \omega^2 & 0 & 1 & \omega^2 & \omega\\
1&\omega & \omega^2 & \omega^2 & \omega & 1 & 0 &  \omega & \omega^2\\
1&\omega^2 & \omega & \omega^2 & \omega & \omega & \omega^2 & 0 & 1 \\
 1&\omega & \omega^2 & \omega & \omega^2& \omega^2 & \omega & 1 & 0 \\
\end{bmatrix}\]
\end{eg}

$ $

$ $

In this paper we have seen an alternative way to generate equiangular frames by taking subsets of groups having certain properties. We are able to show that a lot of real equiangular frames are associated with signature sets and quasi-signature sets. In the case of cube root equiangular frames, so far we have $(9,6)$-cube root equiangular frame and we have shown that it naturally arises from the group of quaternions. The possible $(n,k)$ values for cube root signature matrices listed in \cite{cuberoots} such as $(33,11)$ and $(36,21)$ are still open and could be used for future work. However, after a preliminary check, (45,12) appears to be an equiangular frame, the entries of whose signature matrix are fifth roots of unity.

$ $

{\bf Acknowledgements:} The author would like to thank her Ph.D. advisor, Dr. Vern I. Paulsen for his continuous guidance, support and suggestions. The author would also like to thank Dr. Bernhard Bodmann and Dr. John Hardy for their valuable feedback and discussions.

\end{document}